\documentclass[10pt]{amsart}

\usepackage{enumerate,url,amssymb,  mathrsfs, graphicx, pdfsync}
\newtheorem{theorem}{Theorem}[section]
\newtheorem{lemma}[theorem]{Lemma}
\newtheorem*{lemma*}{Lemma}

\newtheorem{proposition}[theorem]{Proposition}
\newtheorem{corollary}[theorem]{Corollary}

\theoremstyle{definition}
\newtheorem{definition}[theorem]{Definition}

\newtheorem{question}[theorem]{Question}

\theoremstyle{remark}
\newtheorem{remark}[theorem]{Remark}

\numberwithin{equation}{section}

\newcommand{\abs}[1]{\lvert#1\rvert}

\newcommand{\norm}[1]{\lVert#1\rVert}

\newcommand{\C}{\mathbb{C}}

\newcommand{\W}{\mathscr{W}}

\newcommand{\R}{\mathbb{R}}
\newcommand{\X}{\mathbb{X}}

\newcommand{\Y}{\mathbb{Y}}

\newcommand{\x}{\mathfrak{X}}

\newcommand{\Ho}{\mathscr{H}}

\newcommand{\dtext}{\textnormal d}

\newcommand{\onto}{\xrightarrow[]{{}_{\!\!\textnormal{onto\,\,}\!\!}}}
\newcommand{\into}{\xrightarrow[]{{}_{\!\!\textnormal{into\,\,}\!\!}}}

\DeclareMathOperator{\dist}{dist}

\DeclareMathOperator{\loc}{loc}

\DeclareMathOperator{\Div}{div}

\DeclareMathOperator{\supp}{supp}
\def\le{\leqslant}
\def\ge{\geqslant}

\begin{document}

\title[Closures of Sobolev homeomorphisms]{The weak  and  strong closures \\of Sobolev homeomorphisms are the same}

\author[T. Iwaniec]{Tadeusz Iwaniec}
\address{Department of Mathematics, Syracuse University, Syracuse,
NY 13244, USA and Department of Mathematics and Statistics,
University of Helsinki, Finland}
\email{tiwaniec@syr.edu}

\author[J. Onninen]{Jani Onninen}
\address{Department of Mathematics, Syracuse University, Syracuse,
NY 13244, USA}
\email{jkonnine@syr.edu}
\thanks{ T. Iwaniec was supported by the NSF grant DMS-0800416 and the Academy of Finland project 1128331.
J. Onninen was supported by the NSF grant DMS-1001620.}




\begin{abstract}
 Let $\,\X\,$ and $\,\Y\,$  be bounded multiply connected Lipschitz domains in $\,\R^2\,$. We consider the class $\,\Ho_p (\X, \Y)\,$ of homeomorphisms $\,h \colon \X \onto \Y\,$ in the Sobolev space $\,\W^{1,p}(\X,\, \R^2)\,$. We prove that the weak and strong closures of
$\,\Ho_p (\X, \Y)\,$, $\,2\le p< \infty\,$, are equal. The importance of this result to the existence theory in the calculus of variations and anticipated applications to nonlinear elasticity are captured by Theorem \ref{Exisence}.
\end{abstract}

\maketitle

\section{Introduction}
The primary theme of this paper is about homeomorphisms $\, h  :\mathbb X \onto\mathbb Y\,$ between planar domains in the Sobolev space $\,\W^{1,p} (\X, \Y)\,$. Throughout  we use the notation $\,\Ho_p (\X, \Y)\,$ for the class of such mappings and reserve the symbols  $\overline{\Ho}_p (\X, \Y)$ and $\,\widetilde{\Ho}_p (\X, \Y)\,$ for the closures of $\,\Ho_p (\X, \Y)\,$ in strong and weak topology of $\,\W^{1,p}(\X, \Y)\,$, respectively.
\begin{theorem}\label{thm}
Let $\X$ and $\Y$ be $\ell$-connected Lipschitz domains. If $\ell \ge 2$ and $p\ge 2$, then
\begin{equation}\label{west}
\overline{\Ho}_p (\X, \Y) = \widetilde{\Ho}_p (\X, \Y) \, .
\end{equation}
Equality (\ref{west}) also holds for $\,\ell = 1\,$  if $\,p>2\,$,  but fails when $\,p=2\,$.
\end{theorem}
For simply connected domains the equality $\,\overline{\Ho}_2 (\X, \Y) = \widetilde{\Ho}_2 (\X, \Y) \,$  also holds under suitable normalizations of the mappings $\, h  :\mathbb X \onto\mathbb Y\,$. This case is fully resolved in \S\ref{secnat}.

It should be noted that the weak closure $\,\widetilde{\mathscr H_p}(\mathbb X, \mathbb Y)\,$ is none other than the set of all weak limits of homeomorphisms $\, h  :\mathbb X \onto\mathbb Y\,$ in the Sobolev space $\,\mathscr W^{1,p}(\mathbb X, \mathbb Y)\,$. This fact is not obvious but follows from Theorem~\ref{thmain},  see Remark \ref{rem1}.

The interest in Sobolev homeomorphisms comes  naturally from Geometric Function Theory (GFT), Calculus of Variations (CV), and from Nonlinear Elasticity (NE) where prospective applications to elastic plates and thin films gain additional motivation.
In GFT we seek, as a generalization of  the celebrated \textit{Riemann Conformal Mapping Problem}, a homeomorphism $h\in \Ho_p (\X, \Y)$ with smallest $\,p$ -harmonic type energy integrals.  It is certainly unrealistic to expect that the infimum energy within the class $\,\mathscr H_p(\mathbb X\,\mathbb Y)\,$ will always be attained; injectivity is often lost when passing to the weak limit of the minimizing sequence. The best example is the collapsing phenomenon in the minimization of the Dirichlet integral for mappings between circular annuli~\cite{AIM}. It is new, even for the basic case of the Dirichlet energy, that
\begin{equation}\label{dir}
\inf_{h\in \Ho_2 (\X, \Y)}\iint_{\X} \abs{Dh(x)}^2 \, \dtext x = \inf_{h\in \widetilde{\Ho}_2 (\X, \Y)}\iint_{\X} \abs{Dh(x)}^2 \, \dtext x \, .
\end{equation}

At this point let us stress that the Dirichlet energy integrals are quintessential in the study of harmonic mappings~\cite{CIKO, Dub, IKKO, Jo, Job, JS}. In GFT all known examples of mappings with smallest energy appear to be strong limits of homeomorphisms. This observation is now a fact confirmed by Theorem~\ref{thm}.

In the proof of Theorem~\ref{thm} we first consider mappings $\,h\,$ which are weak limits of homeomorphisms $\,h_j \colon \X \onto \Y\,$ in the Sobolev space $\W^{1,p} (\X, \Y)$. They  need not converge strongly. Nevertheless, for the proof of Theorem~\ref{thm}, it seems natural to try suitable corrections in the sequence $\,h_j\,$ to gain strong convergence. This did not work. The strong approximation will be achieved by modifying the limit map $\,h\,$ instead.
 The first observation in our approach is that the weak limit of homeomorphisms $\,h_j\,$  extends continuously to  a monotone map $h \colon \overline{\X} \onto \overline{\Y}$ between the closures. Monotonicity, the concept of C.B. Morrey~\cite{Mor}, simply means that the preimage $h^{-1} (C)$ of a continuum in $\,C \in\overline{\Y}$ is a continuum in $\overline{\X}$. The main ingredient in the proof of Theorem \ref{thm} is a series of step by step replacements of $\,h\,$ by piecewise   $p$-harmonic diffeomorphisms. In making such replacements we heavily rely on the weakly converging sequence of homeomorphisms $\,h_j  : \mathbb X \onto \mathbb Y\,$.
   WE believe that once the monotonicity of $\,h\,$ is established or assumed the sequence $\{h_j\}_{j=1}^\infty\,$ should play no role in the construction of strong approximation of $\,h\,$. This rises even more general question.

\begin{question}\label{Q1}
Let $\X$ and $\Y$ be planar domains. If  $h \colon \overline{\X} \onto \overline{\Y}$ is a monotone Sobolev mapping in $\W^{1,p} (\X, \Y)$, can $h$ be approximated strongly in $\W^{1,p} (\X, \Y)\,$ by homeomorphisms  $\,h_j\, \colon \X \onto \Y$?
\end{question}
Uniform approximation of monotone mappings with homeomorphisms is of great interest in topology. We refer to  Rad\'{o}~\cite{Ra} and Youngs~\cite{Yo, Yo1} for the earliest contributions and to~\cite{Mc, Rab, Whb} for further information. In the Sobolev setting the approximation problem is at the very heart of GFT, and is important in mathematical models of elastic deformations.
 Indeed, Question \ref{Q1} is closely related to  the Ball-Evans question:

\begin{question}\cite{Ba, Ba2}
If $h \in \W^{1,p}(\X, \Y)$ is invertible, can $h$ be approximated in $ \W^{1,p}(\X, \Y)$ by piecewise affine invertible mappings?
\end{question}

This question is of great interest in the study of neohookean energy functionals~\cite{Ba1, BPO1, CL, Ev, SiSp}. The Ball-Evans problem for planar bi-Sobolev mappings which are smooth except for a finite number of points has been resolved in~\cite{Mo}. For analogous questions in the planar H\"older continuous setting see~\cite{BM}, and  for bi-Lipschitz homeomorphisms see~\cite{DP}. The difficulties in saving injectivity in the process of piece-wise linear (equivalently smooth if $n=2,3$) approximation are already recognized in~\cite{Ba,SeS}. In~\cite{IKO2} we showed that every homeomorphism $h \colon \mathbb X \to \mathbb Y$ between planar open sets that belongs to the Sobolev space
$\W^{1,p} (\mathbb X , \mathbb Y)$, $1<p<\infty$, can be approximated in the norm topology of this space by $\mathscr C^\infty$-smooth diffeomorphisms,  See also~\cite{IKO3} for the approximation up to the boundary.

In this paper we actually prove stronger statement than Theorem~\ref{thm}.  Namely, we obtain an approximation by diffeomorphisms which  have the same boundary values as $\,h\,$. Precisely, we have

\begin{theorem}\label{thmain}
Let $\X$ and $\Y$ be $\ell$-connected Lipschitz domains, $2\le \ell < \infty$, and let $h_j \colon \X \onto \Y$ be homeomorphisms converging weakly in $\W^{1,p}(\X, \Y)$, $2 \le p < \infty$, to a mapping $h\in \W^{1,p} (\X, \Y)$.  Then there exists a sequence of $\mathscr C^\infty$-diffeomorphisms
\[h_j^\ast \colon {\X} \onto {\Y}, \qquad h_j^\ast \in h+ \W^{1,p}_\circ (\X, \Y)\]
converging strongly in $\W^{1,p} (\X, \Y)$ to $h$. This also holds  for simply connected Lipschitz domains if $p>2$.
\end{theorem}

Theorem \ref{thmain} has  applications to a study of traction free problems in NE~\cite{Bac, Ba3}. Suppose we are given a well defined energy functional
\begin{equation}\label{energ}
\mathcal E_\mathbb X [h] =  \iint_\mathbb X \mathbf E(x,h, Dh ) \,\,\textnormal{d} x\, , \qquad \textnormal{for Sobolev mappings } \;\;\; h \in \mathscr W^{1,p}(\mathbb X, \mathbb Y) \, .
\end{equation}

 Here the stored energy function  $\,\mathbf E : \mathbb X \times \mathbb Y \times \mathbb R^{2\times 2} \, \rightarrow \mathbb R_+\,$  satisfies the usual Carathe\'{o}dory regularity conditions and coercivity  $\,\mathbf E(x,h, \xi ) \succcurlyeq |\xi|^p\,,$  $\,p\geqslant 2\, $.
However, the key hypothesis on $\,\mathbf E\,$ is the Morrey's quasiconvexity~\cite{Moq}, which just amounts to the lower semicontinuity:
\begin{equation}\label{lowsemi}
\mathcal E_\mathbb X [h] \leqslant \liminf \mathcal E_\mathbb X [h_j] \, , \qquad \textnormal{whenever} \quad h_j \rightharpoonup h \;\;\;\textnormal{ weakly in}\;\;\mathscr W^{1,p}(\mathbb X, \mathbb Y) \, .
\end{equation}

  We shall actually restrict (\ref{energ}) to the subclass $\,\widetilde{\mathscr H_p}(\mathbb X, \mathbb Y)\subset \mathscr W^{1,p}(\mathbb X, \mathbb Y\,)\,$ of weak limits of homeomorphisms.
  Recall our standing assumption (Lipschitz regularity) on the domains $\,\mathbb X\,$ and $\,\mathbb Y\,$ which  is sufficient to have a  continuous extension $\,h :\overline{\mathbb X} \onto \overline{\mathbb Y}\,$ of $\,h \in \widetilde{\mathscr H_p}(\mathbb X, \mathbb Y)\,$ together with a uniform bound of the modulus of continuity by means of the energy. This is immediate from Lemma~\ref{lem211} and Lemma~\ref{lem212}. Now it makes sense to speak of the \textit{partial boundary condition} for $\,h \in \widetilde{\mathscr H_p}(\mathbb X, \mathbb Y)\,$
  \begin{equation}\label{Cond}
  h(x) = h_\circ (x) \;,\;\;\;\textnormal{for}\;\; x \in \mathfrak X \subset \partial \mathbb X
  \end{equation}
  where $\,h_\circ\,$, referred to as a  boundary data, is  a Sobolev homeomorphism in $\,\mathscr H_p(\mathbb X, \mathbb Y)\,$ and $\,\mathfrak X\,$ is a closed subset of $\,\partial \mathbb X\,$. Condition (\ref{Cond})  is void when the set $\,\mathfrak X \subset \partial \mathbb X\,$ is empty. This latter case, being energy-minimal analogue of the Riemann Mapping Problem, became of great interest in GFT.  Another important borderline case, which might be of interest in NE, occurs when $\,\mathfrak X\,=\,\partial \mathbb X\,$. There is, however, a marked difference between our setting and the classical boundary value problems in NE. Here we find the  minimizers in the class $\,\overline{\mathscr H_p}(\mathbb X, \mathbb Y)\,$, which is in the closet  proximity to homeomorphisms. The Lagrange-Euler equations are no longer available~\cite{Ba4, Ba5, SS}; they have to be
 replaced by the inner-variational  equations; also known as  \textit{energy-momentum} or \textit{equilibrium} equations, etc~\cite{Cob, SSe, Ta}. Various results in this direction were obtained in~\cite{BOP, CIKO, IKO4, Mo, Ya}.

 In the general case of (\ref{Cond}) we say that  $\,h\,$ is \textit{traction free} on $\,\partial\mathbb X \setminus \mathfrak X\,$, see~\cite{Bac, Ba3}.  We consider the family $\,\mathscr A = \mathscr A_\mathfrak X( h_\circ ) \subset \mathscr H_p(\mathbb X, \mathbb Y)\,$ of homeomorphisms which satisfy (\ref{Cond}) to inquire into its infimum energy.
  \begin{equation}\label{infEnergy}
  \mathcal E(\mathscr A) = \inf \{ \,\mathcal E_\mathbb X[h] \;; \;\;h \in \mathscr A_\mathfrak X( h_\circ )\,\}
  \end{equation}
  Let $\,\overline{\mathscr A}\,$ and $\,\widetilde{\mathscr A}\,$ denote the closures of $\,\mathscr A\,$  in strong and weak topology of the Sobolev space $\,\mathscr W^{1,p}(\mathbb X, \mathbb Y)\,$, respectively.  The essence of Theorem \ref{thmain} is now captured by the following result
  \begin{theorem} \label{Exisence} Let $\,\mathbb X\,$ and $\,\mathbb Y\,$ be Lipschitz domains. Under the assumption on the energy integral, specified above for equation (\ref{energ}), there exists $\,f \in \overline{\mathscr A} \subset \overline{\mathscr H_p}(\mathbb X, \mathbb Y)\,$ such that
  \begin{equation}
  \iint_\mathbb X \mathbf E(x,f, Df ) \,\,\textnormal{d} x\; =  \mathcal E(\mathscr A)
   \end{equation}
   Here one particular case, when $\,\mathbb X\,$ and $\,\mathbb Y\,$ are simply connected and the exponent $\,p= 2\,$, requires  that the set $\,\mathfrak X\subset \partial \mathbb X\,$ contains at least three distinct points, see Section \S\ref{secnat}. \\In particular, letting  $\,\mathfrak X = \varnothing\,$, we obtain
   \begin{equation}
   \mathcal E_\mathbb X[f] \; =\; \inf_{h \in \mathscr H_p(\mathbb X, \mathbb Y)}\iint_\mathbb X \mathbf E(x,h, Dh ) \,\,\textnormal{d} x\; = \inf_{h \in \widetilde{\mathscr H_p}(\mathbb X, \mathbb Y)}\iint_\mathbb X \mathbf E(x,h, Dh ) \,\,\textnormal{d} x\;
   \end{equation}
  \end{theorem}

Some comments about the proof of Theorem \ref{thmain} merit mentioning here. First we conveniently reduce the problem to the case of domains bounded by circles (Schottky domains). Surprisingly, this takes some efforts to justify such a reduction; namely, continuity (in the Sobolev norm) of the composition operator with a bi-Lipschitz transformation has to be secured.  The main idea of the proof, like in~\cite{IKO2},  is to partition $\,\mathbb X\,$ into small cells $\,U \subset \mathbb X\,$ and replace $\,h\,$ in  $\,U \,$ by a $\,p$ -harmonic diffeomorphism.  However, the novelty lies in the replacements over the boundary cells, see \S\ref{secphar}. We explore topological properties of monotone mappings  to see the geometry of boundary cells; they are certainly simply connected domains.  Then, with the assistance of the weakly converging  sequence $\,\{h_j\}\,$,   we construct  a $\,p$ -harmonic diffeomorphism having the same values on the boundary of a cell as  $\,h\,$. Here, for $\,p > 2\,$,  a $\,p$ -harmonic variant of the classical Hurwitz theorem might be of independent interest.

At this point it may be worth mentioning our definition of a $\,p$-harmonic mapping.
Let $\,h = u + i v\,$ be a complex-valued function in the  Sobolev space $\,\mathscr W^{1,p}(\mathbb X , \mathbb C )\,, 1 < p < \infty\,$. We consider the gradient map $\,\nabla h  = (\nabla u , \nabla v ) \,: \mathbb X \into \mathbb R^2 \times \mathbb R^2 \,$ and its $\,\mathscr L^p$-energy.
\begin{equation}
\mathcal E_\mathbb X [h] = \iint _\mathbb X  |\nabla h(x)|^p \,\textnormal d x \;=\;\iint _\mathbb X  \left (\;|\nabla u(x)|^p \,+\,|\nabla v(x)|^p \;\right )\,\textnormal d x
\end{equation}

The motivation for choosing exactly this form of the $\mathscr L^p$-energy functional comes from the  Lagrange-Euler system which in this case consists of two independent scalar $\,p$-harmonic equations
\begin{equation}\label{p-harm}
\textnormal{div} |\nabla u |^{p-2} \nabla u \; = \;0 \;\;\;\;\;\textnormal{and}\;\;\;\; \textnormal{div} |\nabla v |^{p-2} \nabla v \; = \;0
\end{equation}
\begin{definition}\label{Def}
A complex-valued function $\,h = u + i  v\,$ of Sobolev class $\mathscr W^{1,p}_{\textnormal{loc}} (\Omega, \mathbb C)\,$, $\,1< p < \infty\,$, is called a \textit{$\,p$ -harmonic map} if it satisfies the uncoupled system of equations (\ref{p-harm}).
\end{definition}
This notion is different from what can be found in the literature. We take advantage of Definition \ref{Def} by not actually appealing to the coupled $\,p$ -harmonic system $
\Div |Dh |^{p-2} Dh \; = \;0 \,$, see  \S\ref{sechur}.

\section{Notation and Prerequisites}

\subsection{Domains} Throughout this text $\,\mathbb X\,$ and $\,\mathbb Y\,$ will be, unless otherwise stated, planar Lipschitz domains of finite connectivity $\, 1 \leqslant \ell < \infty\,.$ Thus each boundary $\, \partial \X\, $ and $\,\partial \Y\, $ consists of $\,\ell\,$  disjoint closed Jordan curves which are locally graphs of Lipschitz functions. We reserve the following notation of the ordered $\,\ell$ -tuples of boundary components,
\begin{equation}\label{order}
\begin{split}
(\,\mathfrak X_1, \,\mathfrak X_2,\, ...\,, \,\mathfrak X_\ell\,) \, , \qquad \textnormal{for the components of}  \;\;\partial \X  \\
(\,\Upsilon_1, \,\Upsilon_2,\, ...\,, \Upsilon _\ell\,)  \, , \qquad  \textnormal{for the components of} \;\; \partial \Y
\end{split}
\end{equation}

We shall often distinguish  the outer boundaries of $\X$ and $\Y$. These are the boundaries of  the unbounded components of $\,\mathbb C \setminus \mathbb X\,$ and $\,\mathbb C \setminus \mathbb Y\,$, respectively.   An  $\,\ell$-connected  domain $\,\mathbb X\,$ is said to be \textit{Lipschitz} regular if each boundary component $\,\mathfrak X_1, \mathfrak X_2, ..., \mathfrak X_\ell \,$ is a closed Lipschitz curve, locally (upon rotation) the graph of a Lipschitz function. Let us point out  that any $\,\ell$-connected Lipschitz domain can be transformed via a bi-Lipschitz mapping  $\,\Psi :\mathbb C \onto \mathbb C\,$ onto a domain whose boundaries are circles (Schottky's domain). However, it is not so apparent whether there is such $\,\Psi\,$ whose composition with Sobolev functions represents a continuous (nonlinear) operator. Section~\ref{sectcd} is devoted to constructing this type of bi-Lipschitz transformations.

\subsection{The class $\mathscr H(\mathbb X , \,\mathbb Y)\,$ of homeomorphisms} We shall consider orientation preserving homeomorphisms $\, h\colon \X  \onto \Y \,$.   In general, a homeomorphism $\, h\colon \X  \onto \Y \,$ may not extend continuously  to the closures. Nevertheless, each boundary component of $\,\mathbb Y\,$ is a limit set under $\,h\,$ of one and only one boundary component of $\,\mathbb X\,$.
\subsubsection{Boundary Correspondence} The boundary components of $\,\mathbb X\,$ and $\,\mathbb Y\,$ can be so numbered that the limit set of $\,\mathfrak X_\nu\,$ is $\,\Upsilon_\nu\,$.  Let us assume this numbering and record it as
\begin{equation}\label{Corr}
\;h \colon \mathfrak X_\nu \rightsquigarrow \Upsilon_\nu \;,\;\;\;
\textnormal{hence for the inverse map}  \;\;\;h^{-1} \colon \Upsilon_\nu \rightsquigarrow \mathfrak X_\nu \;,\;\;\nu = 1, 2, ..., \ell\,
\end{equation}
This notation precisely means that for every
$\varepsilon > 0$ there exists $\delta >0$ such that $\textnormal {dist}[ h(x),  \Upsilon_\nu ] < \varepsilon$,  whenever $ \textnormal {dist}[x, \mathfrak X_\nu ]
 < \delta$.

 \begin{definition}
 Given two $\,\ell$-connected domains $\,\mathbb X\,$ and $\,\mathbb Y\,$, each with the specific order of the boundary components as in (\ref{order}). We denote by  $\,\mathscr H(\X, \Y)\, $ the class of orientation preserving homeomorphisms $h \colon \X  \onto \Y$ satisfying~\eqref{Corr}.
\end{definition}
\begin{remark}Without losing any argument in this paper, one could have restricted our definition of   $\,\mathscr H(\X, \Y)\, $ to homeomorphisms in a designated homotopy class.
\end{remark}

\subsection{Monotone mappings} The notion of monotone mappings has originated in the work of C.B. Morrey~\cite{Mor}.

 \begin{definition}[monotonicity]
 A continuous mapping $h \colon \mathbb X \onto \mathbb Y$ between topological spaces is said to be \textit{monotone} if for each $y \in \mathbb Y$  the set $h^{-1}(y) \subset \mathbb X$ is a \textit{continuum}; that is, compact and connected.
 \end{definition}
 Since we will be dealing with continuous monotone mappings between compact subsets of the extended complex plane $\,\widehat{\mathbb C} = \mathbb C \cup \{0\}\,, $ it is appropriate to recall a theorem of  Whyburn~\cite{Mc}.
 \begin{theorem}\label{Whyburn}
 A continuous mapping $f\colon  \mathbf X \onto \mathbf Y$ between compact metric spaces is monotone if and only if $f^{-1}(C)\subset \mathbf X$ is connected for each connected set $C \subset \mathbb Y$.
 \end{theorem}

We also recall a theorem of Kuratowski and Lacher~\cite{Ku, KL}.

\begin{theorem}\label{Kuratowski-Lacher}
   Let $\mathbf X$ and $\mathbf Y$ be compact Hausdorff spaces,  $\mathbf Y$ being locally connected. Suppose we are given a sequence of monotone mappings $f_k \colon  \mathbf X \onto \mathbf Y$ (such are homeomorphisms) converging uniformly to a mapping $f \colon \mathbf X \rightarrow \mathbf Y$, then $f \colon  \mathbf X \onto \mathbf Y$ is monotone.
 \end{theorem}



For further reading about  monotone mappings we refer to~\cite{Mc, Rab, Whb}.

\begin{proposition}\label{Mon}
Let $f \colon \widehat{\C} \onto \widehat{\C}$, $f \colon \C \onto \C$, be a continuous monotone map. If $\;\mathbb K \subset \C\,$ is a continuum that disconnects $\C$ into two components then so is $f^{-1} (\mathbb K)$.
\end{proposition}

\begin{proof}
One needs only observe that there are two types of points in $\widehat{\C} \setminus f^{-1} (\mathbb K)$: the ones that can be path-connected with $\infty$ (they make the unbounded component) and the ones that cannot be path-connected with $\infty$. These latter points are mapped by $f$ onto the bounded component of $\widehat{\C} \setminus \mathbb K$, thus form the remaining (bounded) component of $\widehat{\C} \setminus f^{-1} (\mathbb K)$.
\end{proof}

\subsection{The classes $\mathscr{H}_{p}(\mathbb X, \mathbb Y)\,$ , $ \overline{\mathscr{H}_{p}}(\mathbb X, \mathbb Y)\,$  and  $\,\widetilde{\mathscr{H}_{p}}(\mathbb X, \mathbb Y)$ }
The following classes of mappings of finite energy  will be used throughout.

\begin{itemize}
\item $\mathscr{H}_{p}(\mathbb X, \mathbb Y) = \mathscr{H}(\mathbb X, \mathbb Y)\cap \mathscr W^{1,p}(\mathbb X,\mathbb Y) \,$.
 \item $\overline{\mathscr{H}_{p}}(\mathbb X, \mathbb Y)\,$ is the closure of $\,\mathscr{H}_{p}(\mathbb X, \mathbb Y)\,$ in strong topology of $\,\mathscr W^{1,p}(\mathbb X,\mathbb C)$.
   \item $\widetilde{\mathscr{H}_{p}}(\mathbb X, \mathbb Y)\,$ is the closure of $\,\mathscr{H}_{p}(\mathbb X, \mathbb Y)\,$ in weak topology of $\mathscr W^{1,p}(\mathbb X,\mathbb C)$.
\item $\,\mathscr R^p(\mathbb X)\,$ is the Royden $\,p$-algebra of uniformly continuous functions on $\,\mathbb X\,$ having finite $p$-harmonic energy. The norm in $\mathscr R^p (\X)$ is given by
            $$ \big{\|} h \big{\|}_{\mathscr R^p (\mathbb X)}\, = \,\underset{x\in \mathbb X}{\sup}\,|h(x)| \,+\,\Big(\iint_\mathbb X |\nabla h(x)|^p \textnormal{d} x\,\Big)^{\frac{1}{p}}$$
        \item $\,\mathscr R_0^p(\mathbb X)\,$ is the completion of  $\,\mathscr C^\infty_0 (\mathbb X)\,$ in the above norm.
\end{itemize}

\begin{remark}\label{rem1}
Some reflections concerning the terms weak and strong closures are in order. If  $\, \mathscr H\,$ is a subset of a Banach space $\,\mathscr W\,$, then its weak closure $\,\widetilde{ \mathscr H}\,$ is the smallest subset of $\,\mathscr W\,$ that contains $\, \mathscr H\,$ and is closed under weak convergence. The strong closure $\,\overline{ \mathscr H}\,$, on the other hand,  consists exactly of strong limits of sequences in  $\,\mathscr H\,$.  Thus $\,\overline{\mathscr H} \subset \widetilde{\mathscr H}\,$.  However, in general, when $\,\mathscr H\,$ is unbounded then the set of  all weak limits of sequences in $\,\mathscr H\,$ need not be weakly closed. Consequently,  $\,\widetilde{\mathscr H}\,$ is larger than the set of all weak limits. However, by virtue of Theorem~\ref{thmain}, it is not difficult to see that $\,\widetilde{\mathscr H_p}(\mathbb X, \mathbb Y)\,$ consists exactly of all weak limits  of homeomorphisms $\, h  :\mathbb X \onto\mathbb Y\,$ in the Sobolev space $\,\mathscr W^{1,p}(\mathbb X, \mathbb Y)\,$. It is for this reason that we do not introduce a separate notation for the class of weak limits of homeomorphisms. The interested reader may consult our subsequent paper \cite{IO*}\, for further applications.
\end{remark}

\begin{remark}
Functions in Royden $\,p$-algebra extend continuously up to the boundary, because of uniform continuity. Thus we identify the Royden $\,p$ -algebra as
\begin{equation}
\mathscr R^p(\mathbb X) =  \mathscr C(\overline{\mathbb X}) \cap \mathscr W^{1,p}(\mathbb X)
\end{equation}
This is a commutative Banach algebra with multiplication and addition defined pointwise. Note too, that if  $\,\mathbb X\,$ is sufficiently regular (such are the Lipschitz domains) then  $\,f \in \mathscr R^p(\mathbb X)\,$ belongs to $\, \mathscr R^p_0(\mathbb X)\,$ if and only if $\,f(x)\equiv 0 \,$ on $\, \partial \mathbb X\,$. It is also clear that if a sequence $\,\{f_j\}_{j=1}^\infty \,$ of functions  $\,f_j \in \mathscr R^p(\mathbb X) \,$ converges to $\,f\,$ uniformly on $\,\overline{\mathbb X}\,$ and weakly in $\,\mathscr W^{1,p}(\mathbb X, \mathbb C)\,$ then  $\,f \in \mathscr R^p(\mathbb X) \,$.
\end{remark}

\subsubsection{The class $\widetilde{\Ho}_p (\X , \Y),\,p\geqslant 2\,$}  Since $\,\mathbb X\,$ and $\,\mathbb Y\,$ are Lipschitz domains, their boundaries are neighborhood retracts so we infer from~\cite{IO} that
\begin{lemma}\label{lem211}
Every $h \in \Ho_p (\X, \Y)$ extends as a continuous monotone map $h \colon \overline{\X} \onto \overline{\Y}$. Moreover, the boundary map $h \colon \partial \X \onto \partial \Y$ is also monotone.
\end{lemma}

\begin{lemma}\label{lem212}
We have the uniform bounds of the modulus of continuity.  For $h\in \Ho_2 (\X, \Y)$ and $\ell \ge 2$,
\begin{equation}\label{r2}
\abs{h(x_1)- h(x_2)}^2 \le \frac{C_2(\X, \Y)}{\log \left(e+ \frac{1}{\abs{x_1-x_2}}\right)} \iint_{\X} \abs{\nabla h}^2, \quad x_1, x_2 \in \overline{\X}.
\end{equation}
For any bounded Lipschitz domain $\X \subset \C$ and $h\in \W^{1,p} (\X, \C)$, $p>2$, we have
\begin{equation}\label{r1}
\abs{h(x_1)-h(x_2)}^p \le C_p (\X) \abs{x_1-x_2}^{p-2} \iint_{\X} \abs{\nabla h}^p, \quad x_1, x_2 \in \overline{\X} \, .
\end{equation}
\end{lemma}
 In view of~\eqref{r2} and~\eqref{r1} a sequence $\{h_j\} \subset \Ho_p (\X, \Y)$ that is converging weakly to $h$ actually converges uniformly on $\overline{\X}$.
 Furthermore, such homeomorphisms $h_j \colon \overline{\X} \onto \overline{\Y}$ are monotone so the boundary mappings  $h_j \colon \partial {\X} \onto \partial {\Y}$ are monotone as well. Therefore, the weak limit $h \colon \overline{\X} \onto \overline{\Y}$ is also monotone and $h \colon \partial {\X} \onto \partial {\Y}$ is monotone. It is worth noting that the inequalities~\eqref{r1} and~\eqref{r2} remain valid for  $h\in \widetilde{\Ho}_p (\X, \Y)$.  In particular, we have a continuous imbedding.
\begin{equation}
\widetilde{\Ho}_p (\X, \Y) \subset \mathscr R^p(\mathbb X)
\end{equation}

\subsection{Transition to circular domains}\label{sectcd}

Let $\Y$ be a bounded domain in $\R^2$ and $F \colon \Y \onto \Y'$ a bi-Lipschitz deformation so its inverse $F^{-1} \colon \Y' \onto \Y$ is Lipschitz as well. The induced composition map
\begin{equation}
F_\sharp \colon \W^{1,p} (\X, \Y) \to \W^{1,p} (\X, \Y'), \qquad F_\sharp (u)= F \circ u
\end{equation}
and its inverse
\begin{equation}
F^{-1}_\sharp \colon \W^{1,p} (\X, \Y') \to \W^{1,p} (\X, \Y), \qquad F^{-1}_\sharp (v)= F^{-1} \circ v
\end{equation}
are well defined  nonlinear bounded operators, where $\X$ can be any domain in $\R^2$. It will be  advantageous to transform $\Y$ into a  domain whose boundary components are circles, so-called circular or Schottky domain. We will require  the induced composition operators $\,F_\sharp$ and $F_\sharp^{-1}\,$ to be not only bounded but also continuous. Unfortunately, this is not always the case , see~\cite{Ha1}. Our goal is to construct, a special bi-Lipschitz transformation of  $\,\mathbb Y\,$ into a circular domain.
 \begin{proposition}\label{bipr}
Given any bounded $\;\ell$-connected Lipschitz domain $\Y \subset \R^2$, there exists a bi-Lipschitz map $F \colon \R^2 \onto \R^2$ which takes $\Y$ onto a circular domain. The induced composition operators $F_\sharp \colon \W^{1,p} (\X, \R^2) \to \W^{1,p} (\X, \R^2)$ and the  inverse $ F^{-1}_\sharp \colon \W^{1,p} (\X, \R^2) \to \W^{1,p} (\X, \R^2)$
are continuous for all $1<p< \infty$.
\end{proposition}
Let us comment on some related results. First consider  the Sobolev space $\mathscr W^{1,p}(\X, \R)$ of real-valued functions. If $f \colon \R \to \R$ is Lipschitz the composition $f \circ u$ with $u \in \mathscr W^{1,p} (\X, \R)$ represents a Sobolev function whose gradient can be  defined by the rule
\[
\nabla (f \circ u) = \begin{cases} f' \big(u(x) \big) \, \nabla u \quad & \textnormal{if $f$ is differentiable at $u(x)$}\\
0 & \textnormal{otherwise.}  \end{cases}
\]
The  point is that $f$ is differentiable everywhere except for a set $E \subset \R$ of zero linear measure.  The preimage $\,u^{-1}(E) \subset \X\,$ may have a positive Lebesgue measure. But in this case $\,\nabla u\,$ vanishes on this set. We refer the reader to a paper by Marcus and Mizel~\cite{MM} in which they show that in fact the induced operator
 \begin{equation}\label{bieq1}
f_\sharp \colon  \W^{1,p} (\X, \R) \to  \W^{1,p} (\X, \R) , \qquad f_\sharp (u)= f \circ u
\end{equation}
is continuous.
Now suppose $F \colon \Y \onto \Y'$ is a Lipschitz map that is $\mathscr C^1$-smooth except for a set $E \subset \Y$ of isolated points. Then the induced composition map
 \begin{equation}\label{bieq2}
F_\sharp \colon  \W^{1,p} (\X, \Y) \to  \W^{1,p} (\X, \Y') , \qquad F_\sharp (h)= F \circ h
\end{equation}
 with $h\in \mathscr W^{1,p} (\X, \Y)$ is still well defined. The differential is given by
 \[D(F \circ h)= \begin{cases} DF \big(h(x)\big) \, Dh(x) \;\;,\quad & \textnormal{whenever } h(x) \not \in E \\
 0 \;\;,& \textnormal{otherwise.}
  \end{cases}\]
 The preimage $h^{-1} (E)$ may have positive measure but $Dh(x)$ vanishes on this set anyway. It is not difficult to see that the operator~\eqref{bieq2} is continuous.
\begin{proof}[Proof of Proposition~\ref{bipr}]
We first approximate $\Y$ by polygonal domains. Let $\Upsilon_\nu$ be one of the boundaries of $\Y$. Consider a partition of $\,\Upsilon_\nu\,$ into closed Jordan arcs $\widehat{A_1 A_2},\; \widehat{A_2 A_3}, \;\dots , \;\widehat{A_{n-1} A_n},\; \widehat{A_n A_{n+1}}$, $\,A_{n+1}=A_1\,$, defined by a sequence of consecutive points $A_1,\, A_2, \,\dots , \,A_n \,\in \Upsilon_\nu\,$. Associated with such a partition is a polygonal chain with vertices at $A_1,\, \dots ,\, A_n\,$; that is, a piecewise linear curve $P_\nu$ that consists of the line segments $\overline{A_1A_2}, \; \overline{A_2A_3},\; \dots ,\; \overline{A_{n-1}A_n}, \;\overline{A_nA_1} \,$ connecting the endpoints of the arcs. Although the union $\,\Upsilon_\nu = \bigcup\limits_{i=1}^n \widehat{A_i A_{i+1}}\,$ is a closed Jordan curve the polygonal chain may have points of self-intersection. However, since $\,\Upsilon_\nu\,$ is Lipschitz regular there is $\,\epsilon >0\,$ small enough so that if
\begin{equation}\label{bieq4}
\dist (A_1, A_2) \le \epsilon , \;\dots , \;\dist (A_n, A_1) \le \epsilon
\end{equation}
then the only consecutive segments  $\overline{A_1A_2}, \;\overline{A_2A_3}, \;\dots , \;\overline{A_{n-1}A_n}, \;\overline{A_nA_1}$  intersect, of course  at their common endpoint. Under the condition~\eqref{bieq4} we obtain the boundary of a simply connected domain, called polygon. With $\epsilon > 0 $ sufficiently small we may, and do, ensure that each arc $\widehat{A_1 A_2}, \;\dots, \;\widehat{A_n A_1}\,$,  becomes upon a rotation a graph of a Lipschitz function; points slightly above the graph lie in $\,\Y\,$ and points slightly below the graph lie in $\,\R^2 \setminus \overline{\Y}\,$. Choose and fix one of the arcs $\,\widehat{A_i A_{i+1}}\,$, say the graph of a Lipschitz function $\,y=f(x)\,$, $\,a\le x \le b\,$. Here, for some small positive $\,\delta\,$, we have
\[
\begin{split}
\Omega^+ &= \{(x,y) \colon a \le x \le b , \;\;f(x) \le y \le f(x)+\delta\;\} \subset \overline{\Y}\\
\Omega^- &= \{(x,y) \colon a \le x \le b , \;\;f(x)-\delta \le y \le f(x)\;\} \subset \R^2 \setminus \Y .
\end{split}
\]
With such $\delta$ fixed we may, if needed, further partition the  arc $\widehat{A_i A_{i+1} }$ into a finite number of consecutive subarcs
\[\widehat{A_j A_{j+1} } = \widehat{A_j B_1 } \, \cup \, \widehat{B_1 B_2 } \, \cup \, \dots\,  \cup \, \widehat{B_k A_{j+1} } \]
so that each straight line segment $\overline{A_j B_1}, \;\overline{B_1 B_2}, \;\dots , \;\overline{B_k A_{j+1}}\,$ lies strictly below the graph of the function $y=f(x)+\delta$ and strictly above the graph of the function $y=f(x)-\delta$. However, to simplify the writing we assume without loss of generality  that the arc $\widehat{A_j A_{j+1}}$ already enjoys this property. Thus we have the following region
\[\Omega= \Omega^+ \cup \Omega^- = \{(x,y) \colon \;a \le x \le b , \;\; f(x)-\delta \le y \le f(x)+\delta\,\}\]
and two cross cuts with endpoints $A_j=(a, f(a))$ and $A_{j+1}= (b, f(b))$. One cross cut is the graph of  $y=f(x)$, $a \le x \le b$, and the other cross cut is a straight line segment $\overline{A_j A_{j+1}}$, which we shall view as graph over the interval $[a,b]$ of a linear function $\varphi = \varphi (x)$,
\[\varphi = \frac{(x-a) f(b)+ (b-x)f(a)}{b-a} \;,\quad \,\textnormal{for all}\;\;x \in \R.\]
 Note that $f(x)- \delta < \varphi (x)< f(x)+\delta$ for $a \le x \le b$. It will be convenient to extend $\,f\,$ as a Lipschitz function  in the entire real line $\R$, simply by setting $f(x)= \varphi (x)$ outside the interval $[a,b]$. We are now in a position to define a bi-Lipschitz mapping $\Phi \colon \R^2 \to \R^2$ (associated with the arc $\widehat{A_j A_{j+1} } \subset \partial \Y $) as follows,
 \[\Phi (x,y)= (x, y') \quad \textnormal{ for } - \infty < x < \infty , \quad - \infty < y < \infty\]
where
\[
\begin{split}
y'&=y+ \frac{\varphi (x)-f(x)}{2 \delta} \left(\abs{y-f(x)-\delta} -2 \abs{y-f(x)}+ \abs{y-f(x)+\delta}  \right)\\
&= \begin{cases}
y, & y \ge f(x)+ \delta \\
\frac{f(x)+\delta - \varphi (x)}{\delta}y + \frac{f(x)+\delta}{\delta} \left[\varphi (x)-f(x)\right], \quad & f(x) \le y \le f(x)+\delta\\
\frac{\varphi(x)+\delta - f(x)}{\delta}y + \frac{f(x)-\delta}{\delta} \left[f(x)-\varphi(x)\right], \quad & f(x) -\delta \le y \le f(x)\\
y, & y \le f(x)-\delta
\end{cases}
\end{split}
\]
We see from this later formula that if  $\,x\,$ is fixed the function $y'=y'(x,y)$ is piecewise linear and strictly increasing in $y$. Thus $\Phi$ is a bi-Lipschitz map. Moreover, $\Phi (x,y)= (x,y)$ for all $(x,y) \in \R^2 \setminus \Omega$, which is immediate when $\,y \ge f(x)+ \delta\,$ and $\,y \le f(x)- \delta\,$. On the other hand if $x\not \in (a,b)$ then $\varphi (x)=f(x)$, so $y'=y$ from the first formula. Also note that $\Phi$ takes the arc  $\widehat{A_j A_{j+1} }$ into the line segment connecting $A_j$ and $A_{j+1}$; indeed, for $y=f(x)$ we have $\Phi (x, f(x))= (x, \varphi (x))$. An important feature of this particular map $\Phi \colon \R^2 \to \R^2$ is that it induces a continuous operator
$\Phi_\sharp \colon \mathscr W^{1,p} (\X, \R^2) \onto \mathscr W^{1,p} (\X, \R^2)$. Indeed, given a sequence $h_k (z)= u_k + v_k (z) $ converging strongly in $\mathscr W^{1,p} (\X, \R^2)$ to $h(x)=u(x)+iv(x)\,$, we have
\[
\begin{split}
&\Phi (u_k,\; v_k) = \\ & \left[u_k,\; v_k + \frac{\varphi (u_k)-f(u_k)}{2\delta} \left( \abs{v_k-f(u_k)-\delta}-2 \abs{v_k-f(u_k)}+ \abs{v_k-f(u_k)+\delta}   \right)   \right] \\
&\longrightarrow  \left[u, \;v + \frac{\varphi (u)-f(u)}{2\delta} \left( \abs{v-f(u)-\delta}-2 \abs{v-f(u)}+ \abs{v-f(u)+\delta}   \right)   \right]
 \\
&= \Phi (u,\;v)
\end{split}
\]
 The final step in the construction of the map $\,F \colon \R^2 \to \R^2\,$ consists of dividing the entire boundary $\partial \Y$ into a finite number of sufficiently small arcs $\Gamma_1 , \Gamma_2 , \dots , \Gamma_N \subset \partial \Y$. To each arc $\,\Gamma_i\,$ there corresponds a bi-Lipschitz map $\Phi^i \colon \R^2 \onto \R^2$, $i=1, \dots , N\,$, which takes $\,\Gamma_i\,$ into a line segment connecting the endpoints of $\,\Gamma_i\,$. Note that $\,\Phi^i\,$  is the identity map on all remaining arcs. Therefore, the composition
\[F:= \Phi^1 \circ \dots \circ \Phi^{N} \colon \R^2 \onto \R^2\]
is a bi-Lipschitz map which takes $\,\Y\,$ onto a polygonal domain and the induced operator
\[F_\sharp= \Phi_\sharp^{1} \circ \dots \circ \Phi_\sharp^{N} \colon \mathscr W^{1,p} (\X, \R^2) \to  \mathscr W^{1,p} (\X, \R^2) \]
is continuous. The same applies to the inverse map $F^{-1}$; it induces a continuous operator $F^{-1}_\sharp$ as well. Lastly, we
 compose $F$ with a  bi-Lipschitz map $\,G \colon \R^2 \to \R^2$ which is a $\,\mathscr C^\infty\,$ -diffeomorphism outside the corners of $\,\Y'\,$. This results in a map $\, G \circ F \colon \R^2 \to \R^2\,$ that takes $\Y$ into a  circular domain.
 \begin{remark}
 We may  choose the circular domain  such that the center of the outer boundary lies outside the domain.  This of course requires that $\,\ell \geqslant 2\,$. The centers of the remaining boundary circles certainly lie outside the domain. This additional requirement will later help us to perform the reflections  about the boundary circles. The case $\, \ell = 1\,$  will be treated differently in \S\ref{seclas}.
 \end{remark}
 \begin{remark}
 An analogous transition from the domain $\, \mathbb X\,$ into a circular domain $\, \mathbb X'\,$, via a bi-Lipschitz map $\,T :\,\mathbb C\,\onto \mathbb C\,,\; T : \mathbb X \onto \mathbb X'\,$, presents no difficulty. The induced composition operator

\[
T^\sharp \colon \W^{1,p} (\X', \Y) \to \W^{1,p} (\X, \Y)
\;\;\;\quad\;\;T^\sharp u  \;= \; u\circ T
\]
and its inverse are always continuous,  $1<p< \infty$.
 \end{remark}
\end{proof}
\begin{center}
From now on, unless otherwise stated,  both $\mathbb X\,$ and $\,\mathbb Y\,$ are \\circular domains containing no centers of the boundary circles.

\end{center}

\subsection{Extension to $\widehat{\C}$}\label{secext}

For convenience and also for easy references, we shall extend  $\,h \colon \overline{\X} \onto \overline{\Y}\,$ to a monotone mapping of $\widehat{\C}= \C \cup \{\infty\}$ onto itself.  Here, with the usual convention concerning the point $\infty$, the one point compactification $\widehat{\C} = \C \cup \{\infty \}$ will be identified (topologically) with $\mathbb S^2$. Recall that we are working under the assumption  that the  boundary components $\x_1, \dots , \x_\ell \subset \partial \X $ and $\Upsilon_1, \dots , \Upsilon_\ell \subset \partial \Y$ are circles. To each circle $\x_i = \{z
\colon \abs{z-z_i}=r_i\}\,$ there corresponds the reflection map
\[     \varphi_i \colon \widehat{\C} \to \widehat{\C} \;\;\;,\;\;\;\varphi_i (z) = \frac{z-z_i}{\abs{z-z_i}^2}\;r_i^2 + z_i \, , \qquad i =1, \dots , \ell \, .\]
Denote by $\X_i = \varphi_i (\X)\,$ the  reflected domains. These domains together with $\overline{\X}$ furnish a circular domain of connectivity $\ell (\ell -1)$
\[\X_+ = \overline{\X} \cup \X_1 \cup \dots \cup \X_\ell \Supset \X \, .\]
The boundary components of $\X_+$ are denoted by
\[\x_i^j = \varphi_i (\x_j), \qquad i \not = j\, .\]
Similarly to $\X$, we consider the boundary circles $\Upsilon_i$ of  $\,\Y\,$ and define the corresponding reflections  $\psi_i \colon \widehat{\C} \to \widehat{\C}$. We obtain  mutually disjoint domains $\Y_i = \psi_i (\Y)$, $i=1, \dots , \ell$, which together with $\overline{\Y}$ furnish a circular domain
\[\Y_+ = \overline{\Y} \cup \Y_1 \cup \dots \cup \Y_\ell  \Supset \Y \, . \]
The boundary of $\Y_+$ consists of  $\ell (\ell -1)$  circles $\Upsilon_i^j = \psi_i (\Upsilon_j)$, $i \not = j$. Now, let $h \colon \overline{\X} \onto \overline{\Y}$ be any continuous mapping that takes each $\x_i \subset \partial \X$ into the corresponding boundary component $\Upsilon_i \subset \partial \Y\,$. We also assume that in this correspondence $h$ takes the outer boundary of $\X$ into the outer boundary of $\Y$.

\subsubsection{The extension $\,\widehat h\,: \overline{\X}_+ \onto \overline{\Y}_+$}
There is a natural way to extend $h$ continuously to a map of $\overline{\X}_+$ onto $\overline{\Y}_+$. The notation $\widehat{h} \colon \overline{\X}_+ \onto \overline{\Y}_+$  will be reserved for such an extension of $\,h \colon \overline{\X} \to \overline{Y}\,$.  By definition, the map $\widehat{h} \colon \overline{\X}_+ \onto \overline{\Y}_+$ takes the reflection of $\X$ into the corresponding reflection of $\,\Y\,$
\[ \quad \widehat{h}= \psi_i \circ h \circ \varphi_i \colon \overline{\X}_i \to \overline{\Y}_i \quad \textnormal{ for } i=1,2, \dots , \ell \, . \]
It should be noted that if $h \colon \overline{\X} \to \overline{\Y}$ is monotone and continuous, then so is the mapping $\,\widehat{h} \colon \overline{\X}_+ \onto \overline{\Y}_+\,$.  Since in our applications the boundary mappings $\,h \colon \x_i \onto \Upsilon_i\,$ are monotone it follows  that
$\,\widehat{h} \colon \x_i^j \onto \Upsilon_i^j\,$, $i \not = j$, are monotone as well.  Furthermore, whenever $h \colon \X \onto \Y$ is a homeomorphism, the extended mappings $\,\widehat{h} \colon \X_i \onto \Y_i$, $i=1, \dots, \ell$, are homeomorphisms.

We will explore the Royden $p$-algebra. Note that if $h \in \mathscr C (\overline{\X}, \overline{\Y}) \cap \W^{1,p} (\X, \Y)= \mathscr R^p (\X, \Y)\,$ and $\,h \colon \x_i \onto \Upsilon_i$, for $i=1, \dots ,\ell\,,$ then $\widehat{h} \in \mathscr C (\overline{\X}_+, \overline{\Y}_+) \cap \W^{1,p} (\X_+, \Y_+)= \mathscr R^p (\X_+, \Y_+)$. We have a uniform bound of the $p$-harmonic energy
\[\iint_{\X_+} \abs{\nabla h}^p \le C_p (\X, \Y) \iint_{\X} \abs{\nabla h}^p\, .\]

\subsubsection{Weak continuity of the Jacobian determinant}  If a sequence of mappings $\,f_k \in \mathscr W^{1,p}(\mathbb X, \mathbb C )\,$, $p>2$, converges to $\,f\,$ weakly in $\,\mathscr W^{1,p}(\mathbb X)\,$  then the Jacobian determinants $\, J(x, f_k)\,$ converge to $\, J(x, f)\,$ weakly in $\,\mathscr L^1(\mathbb X)\,$. We write it as
\begin{equation}\label{weakJac}
\iint_\mathbb X\varphi(x) J(x, f_k) \,\textnormal d  x  \;\;\longrightarrow  \iint_\mathbb X \varphi(x) J(x, f) \,\textnormal d  x\;,\;\;\;\textnormal{for every}\;\;\varphi \in \mathscr L^\infty(\mathbb X)
\end{equation}
However, if  $\,p = 2\,$, this property is invalid for several reasons. The best example to illustrate is a  sequence of M\"{o}bius self-homeomorphisms of the unit disk converging to a constant map, see \S\ref{seclas} formula~\eqref{eqmo}. Note that in this example all the Jacobians are nonnegative, yet the integrals (\ref{weakJac}) fail to converge. The situation is quite different if the mappings in question admit $\,\mathscr W^{1,2}$ - extension beyond $\,\overline{\mathbb X}\,$ with a nonnegative Jacobians. Precisely, we have
\begin{lemma}
 Suppose $\,\mathbb X\,$ is compactly contained in a domain $\,\mathbb X_+\,$; such is our Schottky domain $\X$ and its extension $\X_+$ by reflections. Let a sequence of mappings $\,\widehat f_k : \mathbb X_+ \into  \mathbb C \,$ with nonnegative Jacobians  be bounded in $\,\mathscr W^{1,2}(\mathbb X_+, \mathbb C )\,$ and converge to $\,f\,$ weakly in $\,\mathscr W^{1,2}(\mathbb X, \mathbb C)\,$. Then (\ref{weakJac}) holds.
\end{lemma}
As a particular case, if $\,U \subset \mathbb X\,$ is a measurable set we obtain, upon setting $\,\varphi = \chi_{_U}\,$,  the following variant of (\ref{weakJac})
\begin{equation}\label{averJac}
\iint_U J(x, f_k) \,\textnormal d  x  \;\;\longrightarrow  \iint_U  J(x, f) \,\textnormal d  x\;
\end{equation}
In other words the averages of $\,J(x, f_k)\,$ over any measurable set $\,U\,$ converge to the average of $\,J(x, f)\,$. The proof is not straightforward, but relies on the well known $\,\mathscr L \log \mathscr L$ - integrability of nonnegative Jacobians ~\cite{Mu}. The interested reader is referred to~\cite{CLMS, IOh} for further reading about Jacobians.

\subsubsection{An auxiliary extension $\,\widehat h\,: \overline{\C} \onto \overline {\C}\,$}\label{secaux}
Proceeding in this direction, we shall further extend $\widehat{h} \colon \overline{\X}_+ \onto \overline{\Y}_+$ to the entire Riemann sphere $\widehat{\C}\,$,   again denoted by $\widehat{h} \colon \widehat{\C} \onto \widehat{\C}\,$. Let us point out that we will  need this  extension only to provide easy references to the theory of monotone mappings on $\,\mathbb S^2\,$; for example,  Theorems~\ref{Whyburn} and~\ref{Kuratowski-Lacher}. Let $\,\X_i^j\,$ and $\,\Y^j_i\,$, $\,i \not = j\,$, denote the disks enclosed by the circles $\,\x_i^j\,$ and  $\Upsilon_i^j\,$, respectively. There is one exception to this notation; namely, the case when the circles $\x_i^j \subset \partial \X_+$ and  $\Upsilon_i^j \subset \partial \Y_+\,$ are the outer boundaries. In this case $\,\X_i^j\,$ and $\,\Y_i^j\,$  stand for the complements of the disks. We then extend  each boundary map $\,\widehat h \colon \x_i^j \onto \Upsilon_i^j$, $i \not = j\,$, in a radial fashion into   $\,\X_i^j\,$ and $\,\Y^j_i\,$, and continue to write $\,\widehat{h} \colon \X_i^j \onto \Y_i^j$, $i \not = j\,$.   As an illustration, the radial extension into the disks  $\,\x_i^j =\{r e^{i \theta} \colon 0 \le \theta < 2 \pi\}$ and $\Upsilon_i^j = \{Re^{i \theta} \colon 0 \le \theta < 2 \pi\}\,$ takes the form
\[     \widehat{h} \colon \X_i^j \to \Y_i^j  \;\;, \;\;\;\;\;\;\;\; \widehat{h}(\rho e^{i \theta})= \frac{\rho}{r}\, \widehat h(r e^{i \theta})\]
for $0 \le \rho \le r$, and for $r \le \rho \le \infty$ in case of the complements of the disks.

There will be no need for Sobolev regularity of $\,\widehat{h}\,$ outside $\X_+$. Of course, continuity  of $\widehat{h} \colon \widehat{\C} \to \widehat{\C}$ is understood with respect to the choardal metric induced by the stereographic projection of $\widehat{\C}$ onto $\mathbb S^2$.

Observe that a uniform convergence of continuous mappings $\,h_j \colon \overline{\X} \onto \overline{\Y}$, $h_j \colon \x_i \to \Upsilon_i$, $i=1, \dots, \ell$, $j=1,2, \dots$ to a mapping $\,h \colon \overline{\X} \onto \overline{\Y}$
yields uniform convergence of $\,\widehat{h}_j \colon \widehat{\C} \onto \widehat{\C}\,$ to $\,\widehat{h} \colon \widehat{\C} \onto \widehat{\C}$. If, in addition, all the mappings $\,h_j\,$ are monotone then so is the limit mapping $\,\widehat{h}\,$.
\subsection{The $\,p$ -harmonic Dirichlet problem} There are  two common settings of the Dirichlet problem. The classical one, with a continuous boundary data, combines Perron method and Winer's criterion of regular points on the boundary of a domain.
In the variational approach, on the other hand, one seeks to minimize the energy integral over the class of functions in $\; u+ \mathscr W_\circ^{1,p}(\Omega)\,$, where $\;u \in  \mathscr W^{1,p}(\Omega)$ is viewed as the boundary data. Even when these two different settings are well defined the question whether they lead to the same solution  involves a delicate analysis of the boundary of the domain. Strangely, in the widely spread theory of the Dirichlet integral,   explicit statements concerning simply connected domains appear to be rare in the literature. The equivalence of these two settings is vital in our approach.

\begin{theorem}\label{winlem}
Let $\Omega \subset \mathbb R^2\,$ be a bounded simply connected domain  and $u\in \mathscr C (\overline{\Omega}) \cap \mathscr W^{1,p}(\Omega)$, $1<p< \infty$. Then there exists unique function $\tilde{u}\in \mathscr C (\overline{\Omega}) \cap \mathscr W^{1,p}(\Omega)$ that is $p$-harmonic in $\Omega$ and equals $\, u\,$ on $\,\partial \Omega\,$. Furthermore, $\tilde{u}\in u+ \mathscr W_\circ^{1,p}(\Omega)$ and
\[\mathcal E_p [\tilde{u}] \le \mathcal E_p[u]\]
equality occurs if and only if $\tilde{u} \equiv u$.
\end{theorem}
\begin{proof}We shall not give all details for the proof. Nevertheless, it is worth remarking (because it is not obvious) that the variational solution $\tilde{u}\,$ extends continuously up to the boundary. This is because each boundary point of a planar simply connected domain is a regular point for
the $p$-Laplace operator $\Delta_p$~\cite[p.418]{Ha}. See~\cite[6.16]{HKMb} for the full discussion of boundary regularity and relevant notion of capacities. In particular, we refer the reader to~\cite[Lemma 5.3]{HKM}, see also~\cite[Lemma 4.1]{KM} and~\cite[Lemma 2]{Le} for a capacity estimate that applies to simply connected domains.
\end{proof}
\subsection{Univalent $p$-harmonic extension}
The celebrated Rad\'{o}-Kneser-Choquet theorem asserts that a harmonic function $\,h : \Omega \rightarrow \mathbb C\,$ in a Jordan domain, which extends continuously as a homeomorphism of $\,\partial \Omega\,$ onto a closed convex curve $\,\Gamma\subset \mathbb C\,$,  is a $\,\mathscr C^\infty$ -diffeomorphism of $\, \Omega\,$ onto the bounded component of $\,\mathbb C \setminus \Gamma\,$.  In fact our proof of Theorem~\ref{thmain} in case $\,p=2\,$ relies heavily on the theory of harmonic mappings. An excellent reference is~\cite{Dub}.  Similar arguments will apply to the general case of $\,p>2\,$. However, new ingredients (the $\,p$ -harmonic variant of Hurwitz theorem) will be needed. Meanwhile, let us call upon  the following variant of the Rad\'{o}-Kneser-Choquet theorem from the work of Alessandrini and  Sigalotti~\cite{AS}, see  Theorem 5.1 therein.
\begin{theorem}\label{AS}  Suppose that $\,\Omega \subset \mathbb C\,$ is a simply connected Jordan domain, G a bounded convex domain and $\,h \in \mathscr C(\overline{\Omega}) \cap \mathscr W^{1,p}_{\textnormal{loc}}(\Omega)\,$ a $\,p$ -harmonic map that takes $\,\partial \Omega\,$ homeomorphically onto $\,\partial G\,$. Then $\,h : \Omega \onto  G\,$ is a $\,\mathscr C^\infty$ -diffeomorphism.
\end{theorem}

\section{Proof of Theorem \ref{thmain} for multiply connected domains}
Some additional prerequisites, more specific to our proof, are in order. We present them in the following subsections.

\subsection{Squares and cells}\label{secsqu}
Let for a moment $\widehat{h} \colon \widehat{\C} \onto \widehat{\C}$ be a general continuous monotone map, $\widehat{h} \colon \C \to \C$, $\widehat{h}(\infty)=\infty$. We shall work with an open square $\mathbb Q \subset \C$  and its preimage $\Omega = \widehat{h}^{-1}(\mathbb Q)$.  The set $\Omega$ is called  a {\it cell}, which is surely a domain. Its complement $\widehat{\C} \setminus \Omega$, being equal to $\widehat{h}^{-1} (\widehat{\C} \setminus \mathbb Q)$, is also connected. In view of unicoherence of $\widehat{\C}$, the boundary  $\partial \Omega = \overline{\Omega}\cap (\widehat{\C} \setminus \Omega)$ is connected. Thus $\Omega$ is a simply connected domain. As for the preimage $\widehat{h}^{-1} (\partial \mathbb Q)$ a caution is required. Although it is straightforward that $\partial \widehat{h}^{-1}(\mathbb Q) \subset \widehat{h}^{-1} (\partial \mathbb Q)$ the set $\mathcal C :=\widehat{h}^{-1} (\partial \mathbb Q)$ can be significantly larger than $\partial \widehat{h}^{-1}(\mathbb Q)$. We note that $\mathcal C$ is a continuum disconnecting $\widehat{\C}\,$ into two components, see Proposition \ref{Mon} . Also note that the bounded component of $\,\mathbb C \setminus \mathcal C\,$ equals  $\,\Omega=\widehat{h}^{-1} (\mathbb Q)$.

Next consider an increasing sequence $\mathbb Q_1 \Subset \mathbb Q_2 \Subset \dots \Subset \dots \mathbb Q$ of open squares, $\mathbb Q_n =\lambda_n \mathbb Q$, $0< \lambda_1 < \lambda_2 \dots \to 1$, so $\bigcup_{n=1}^\infty \mathbb Q_n = \mathbb Q$.  Hereafter, the notation $\lambda \mathbb Q$ stands for a square with the same center as $\mathbb Q$ but $\lambda$-times smaller than $\mathbb Q$. We will be dealing with the induced cells $\Omega_1 \Subset \Omega_2 \Subset \dots \Subset \dots \Omega\,$, where $\,\Omega_n = \widehat{h}^{-1} (\mathbb Q_n)\,$, so  $\,\bigcup_{n=1}^\infty \Omega_n = \Omega\,$. Choose  one of the boundaries $\partial \mathbb Q_n\,$, $n=2,3, \dots\,$.  This is a  Jordan curve which separates $ \overline{\mathbb Q}_{n-1}$ from $\partial \mathbb Q_{n+1}$. By virtue of monotonicity of $\widehat{h} \colon \widehat{\C} \onto \widehat{\C}$ we see that $C_n := \widehat{h}^{-1} (\partial \mathbb Q_n)$ is a continuum in $\Omega_{n+1}\,$. This continuum  separates $\partial \Omega_{n+1}$ from $\overline{\Omega}_{n-1}$.  Precisely, $\,C_n \subset \Omega_{n+1}\,$ and the bounded component of $\widehat{\C} \setminus C_n$  contains $\overline{\Omega}_{n-1}$,  because no point in $\Omega_n = \widehat{h}^{-1} (\mathbb Q_n)$ can be path-connected with $\infty$ within the open set  $\widehat{\C} \setminus C_n$.

Now supposed we are given a sequence $\{\widehat{h}_j\}^\infty_{j=1}$ of continuous monotone mappings $\widehat{h}_j \colon \widehat{\C} \to \widehat{\C}\,$, such that $\,\widehat{h}_j \colon \C \to \C\,$, and $\,\widehat{h}(\infty)=\infty\,$. We assume that this sequence converges uniformly to $\widehat{h} \colon \widehat{\C} \to \widehat{\C}$. Denote the preimages of $\partial \mathbb Q_n$ under $\widehat{h}_j\,$ by $\,C_n^j=\widehat{h}_j^{-1}(\partial \mathbb Q_n)\,, j = 1,2, ...\;$. These are continua disconnecting $\widehat{\C}$ into two components. The sets
\begin{equation}\label{eq521}
\Omega_n^j := \widehat{h}_j^{-1} (\mathbb Q_n)
\end{equation}
are the bounded components of $\widehat{\C} \setminus C_n^j$. Just the uniform convergence $\widehat{h}_j \rightrightarrows \widehat{h}$ yields
\[\lim\limits_{j \to \infty} \sup\limits_{x\in C_n^j} \dist (x, C_n)=0 \, .\]
Therefore, for sufficiently large $\,j\,$, say $\,j \ge j_n$, the continua $C_n^j$ separate $\overline{\Omega}_{n-1}$ from $\partial \Omega_{n+1}$. We record this fact as
\begin{equation}\label{eq52}
\overline{\Omega}_{n-1} \subset \Omega_n^j \subset \overline{\Omega}_n^j \subset \Omega_{n+1}\, , \quad n \ge 2 \textnormal{ and } j \ge j_n\, .
\end{equation}

\subsubsection{Cutting a cell}\label{secspl} From now on the continuous monotone mapping $\widehat{h} \colon \widehat{\C} \to \widehat{\C}$   will be the one that is obtained from  the extension of $h \colon \overline{\X} \onto \overline{\Y}$ between circular domains, as in \S\ref{secext}. Fix one of the boundary circles $\x_1, \dots , \x_\ell \subset \partial \X$, say $\x_i$ for some $i=1, \dots , \ell$, and recall that $\Upsilon_i =h(\x_i)$ is the corresponding boundary circle in $\partial \Y$. We also recall the notation $\mathbb X_i$ for a reflection of $\X$  about $\x_i$, and $\Y_i$  for a reflection of $\Y$  about $\Upsilon_i$. Now consider an open square $\mathbb Q$ that intersects $\Upsilon_i$ along an open arc $\gamma= \Upsilon_i \cap \mathbb Q$. We assume that $\mathbb Q$ is small enough so it lies entirely in $\Y \cup \Upsilon_i \cup \Y_i$.  The  arc $\gamma$ is a cross-cut of $\mathbb Q$; it cuts $\mathbb Q$ into two connected subdomains $\mathbb Q \cap \Y$ and $\mathbb Q \cap \Y_i$. We shall now discuss in some detail  an analogous cross-cut of the cell $\Omega=\widehat{h}^{-1}(\mathbb Q)$. To this end we notice that the boundary map $h \colon \x_i \onto \Upsilon_i$ defines an open subarc of $\,\x_i\,$ which lies in $\,\Omega\,$
\[\beta = \{x \in \x_i \colon h(x)\in \gamma\}  \, .\]
 The endpoints of $\beta$ belong to $\partial \Omega$. Indeed, if $x$ is an endpoint of $\beta$ then there are points $x_\nu \in \beta \subset \Omega$ which converge to $x \not \in \beta$,  thus $h(x_\nu) \in \gamma$. Passing to the limit we obtain $h(x)\in \overline{\gamma} \, \setminus \,  \gamma \subset \partial \mathbb Q$. Therefore, $x \not \in h^{-1} (\mathbb Q)= \Omega$, meaning that $x\in \partial \Omega$.

This is surely a geometric folklore that an arc $\beta \subset \Omega$ (in a simply connected domain $\Omega$) whose endpoints lie in $\partial \Omega$ splits $\Omega$ into two connected subdomains; namely,
 \begin{equation}U= \X \cap \Omega \quad \textnormal{ and } \quad \X_i \cap \Omega\,   \end{equation}
This latter subdomain will be of no interest to us.

\subsection{Hurwitz's theorem for $p$-harmonic mappings}\label{sechur} Let us begin withg the classical Hurwitz theorem:
\begin{theorem}\label{huthm1}
If a sequence of holomorphic functions $f_n \colon \Omega \to \C$ converges $c$-uniformly to a holomorhic function $f \colon \Omega \to \C$, and $f_n(z) \not = 0$ for all $z\in \Omega$ and $n=1,2, \dots$, then $f$ is either equal  identically to zero or $f(z) \not = 0 $ for all $z\in \Omega$.
\end{theorem}
As a particular consequence of Hurwitz's theorem one obtains
\begin{theorem}\label{huthm2}
If a sequence of conformal mappings $\varphi_n \colon \Omega \to \C$ converges $c$-uniformly to $\varphi \colon \Omega \to \C$, then $\varphi$ is either constant or conformal.
\end{theorem}
In our arguments we will explore  an adaptation of these two results to quasiregular mappings. It is immediate from Stoilow factorization of quasiregular mappings~\cite{Ahb, AIMb} that
\begin{corollary}\label{qrHur}
 In Theorem~\ref{huthm1} holomorphic functions can be replaced by $K$-quasiregular mappings while in Theorem~\ref{huthm2} conformal mappings can be replaced by $K$-quasiconformal homeomorphisms.
 \end{corollary}

In this section we aim to prove the following analogue of Hurwitz's theorem in a  $\,p$-harmonic setting.
\begin{theorem}\label{p-Hurwitz}
If a sequence of $p$-harmonic orientation-preserving homeomorphisms $\varphi_n \colon \Omega \to \C$ converges $c$-uniformly to $\varphi \colon \Omega \to \C$, then either $\varphi$ is a $p$-harmonic homeomorphism (actually $\mathscr C^\infty$-diffeomorphism) or $J_\varphi (z) \equiv 0$ in $\Omega$.
\end{theorem}
\begin{proof}
For basic properties of quasiregular mappings in relation to $\,p\,$-harmonic functions  we refer the reader to~\cite{BI} and for further reading to~\cite{DI}.
 Recall from~\cite{IM} that every $p$-harmonic function $u\colon  \Omega \to \R$ actually belongs to $\mathscr C^{k, \alpha}_{\loc} (\Omega)$, where the largest integer $k \ge 1$ and the H\"older exponent $\alpha \in (0,1]$ are determined through  the equation
\[k+\alpha = \frac{7p-6+\sqrt{p^2+12p-12}}{6p-6}\,  .\]
Thus, regardless of the exponent $p$, we have $u \in \mathscr C^{1, \alpha}_{\loc}(\Omega)$
with $\,\alpha >1/3\,$. The complex gradient
\begin{equation}\label{heq1}
f:= u_z = \frac{1}{2} \left(u_x-iu_y \right), \qquad z=x+iy
\end{equation}
 is a $K$-quasiregular mapping
 \[K\le \begin{cases} p-1 \qquad & \textnormal{if } p \ge 2 \\
(p-1)^{-1} &  \textnormal{if } 1 < p \le 2 .
  \end{cases}\]
 Actually, $\,f\,$  satisfies the quasilinear elliptic equation
 \begin{equation}\label{heq2}
 \frac{\partial f}{\partial \bar z} = \frac{2-p}{2p} \left[ \frac{\overline{f}}{f}\,  \frac{\partial f}{\partial z} + \frac{f}{\overline{f}}\,  \frac{\overline{\partial f}}{\partial z}  \right]
 \end{equation}
 In particular, every sequence $\{u^n\}^\infty_{n=1}$ of $p$-harmonic functions converging $c$-uniformly on $\Omega$, actually converges in $\mathscr C^{k, \alpha}_{\loc} (\Omega)\,$. Then, by Corollary \ref{qrHur} , if the complex gradients $f^n=u_z^n$ are nowhere vanishing in $\Omega$ then $f=u_z$ is either identically zero or is also a nowhere vanishing function. Moreover, $f^n \to f\,$ uniformly on compact subsets together with all derivatives. Now consider the $p$-harmonic homeomorphisms
\[\varphi_n:= u^n +i\, v^n \colon \Omega \to \C\]
 and their limit
\[\varphi := u +i\, v \qquad \varphi \colon \Omega \to \C.\]
Theorem~\ref{AS} tells us
 that $\varphi_n$ are diffeomorphisms of positive Jacobian determinant, $\,J_{\varphi_n}(z)=u_x^nv_y^n -u_y^nv_x^n > 0\;$. Therefore, the complex gradients,
 \[f^n (z):= u^n_z \quad \textnormal{ and } \quad g^n (z):= v^n_z  \]
 do not vanish in $\Omega\,$. In view of Corollary \ref{qrHur} the limit functions $\,f(z)=u_z \,$ and $\,g(z)=v_z \,$ do not vanish as well, unless they are identically equal to zero. But in this latter case we would have $J_\varphi = u_xv_y-u_yv_x \equiv 0$, proving  Theorem~\ref{p-Hurwitz}. Therefore, let us  assume that neither $\,f \neq 0\,$ nor $\,g\neq 0\,$, everywhere. We aim now to show that  also $\,J_\varphi (z) \not =0$ everywhere in $\Omega$. Let us notice in advance that once the inequality $J_\varphi (z) \not = 0$ is established, the map $\varphi$ will be a local diffeomorphism. Even more, since this map is a $c$-uniform limit of homeomorphisms, it must be a global homeomorphism,  by elementary topological considerations. \\
 Striving for a contradiction, assume that $J_\varphi (z_\circ)=0$ at some point $z_\circ \in \Omega$, or equivalently,
 \[\alpha f(z_\circ)+ \beta g(z_\circ)=0 \quad \textnormal{ for some real numbers } \alpha, \beta \not =0.\]
 We consider the complex functions
 \[F^n (z)= \alpha f^n (z)\;+ \;\beta\, g^n (z)\]
 and their limit
 \[F (z)= \alpha f (z)\;+ \;\beta\, g (z)\, ,  \qquad F(z_\circ) = 0 \, .\]
 Given any real coefficients $\alpha , \beta \not = 0$ we shall construct elliptic first order system of partial differential equations for $F^n$. These equations will be of aid in determining that $F(z) \not =0$ in $\Omega$, contradicting the equality $F(z_0)=0$. The derivation of the equation goes as follows.

 Using~\eqref{heq2} and the analogous equation for $g$ we compute
 \[
 \begin{split}
 F^n_{\bar z} &= \frac{2-p}{2p} \left[\frac{\overline{f^n}}{f^n} \frac{\partial f^n}{\partial z} + \frac{f^n}{\overline{f^n}} \frac{\overline{\partial f^n}}{\partial z}  \right]\, \alpha+ \frac{2-p}{2p} \left[\frac{\overline{g^n}}{g^n} \frac{\partial g^n}{\partial z} + \frac{g^n}{\overline{g^n}} \frac{\overline{\partial g^n}}{\partial z}  \right]\, \beta\\
 &=  \frac{2-p}{2p} \left[\frac{\overline{g^n}}{g^n} \frac{\partial F^n}{\partial z} + \frac{g^n}{\overline{g^n}} \frac{\overline{\partial F^n}}{\partial z}  \right] \\ & \; + \frac{2-p}{2p} \left[\left(\frac{\overline{f^n}}{f^n}-\frac{\overline{g^n}}{g^n}\right) \frac{\partial f^n}{\partial z} + \left(\frac{f^n}{\overline{f^n}} - \frac{g^n}{\overline{g^n}}\right) \frac{\overline{\partial f^n}}{\partial z}  \right]\, \alpha
 \end{split}
 \]
 Here we have
 \[\frac{\overline{f^n}}{f^n}-\frac{\overline{g^n}}{g^n} = \frac{(\alpha f^n + \beta g^n)\overline{f^n} - (\alpha \overline{f^n} + \beta \overline{g^n}){f^n}  }{\beta\, f^n g^n}\]
 and hence
 \[\left|\frac{\overline{f^n}}{f^n} - \frac{\overline{g^n}}{g^n}  \right| \le \frac{2\, \abs{F^n}}{\abs{\beta} \abs{g^n}} \, . \]
 Thus we have the following first order (elliptic) inequality
 \[\left| \frac{\partial F^n}{\partial \bar z} \right| \le \left|1-\frac{2}{p}\right| \, \left| \frac{\partial F^n}{\partial z}\right| \;\;+\;\; \left|1-\frac{2}{p}\right| \frac{2
 \abs{\alpha}}{\abs{\beta} \abs{g^n}} \,  \abs{F^n} .\]
 This inequality can be viewed, equivalently, as a linear equation
 \begin{equation}\label{heq3}
 \frac{\partial F^n}{\partial \bar z} = \mu^n (z)  \frac{\partial F^n}{\partial  z} + A^n (z) F^n \, , \quad \textnormal{ a.e. in } \Omega \, .
 \end{equation}
 with complex measurable coefficients satisfying:
 \[
 \begin{split}
 \abs{\mu^n (z)} & \le k := \left| 1-\frac{2}{p} \right|\; <\;1\\
 \abs{A^n (z)} &\le 2k \left| \frac{\alpha}{\beta}\right|\, \left|\frac{1}{g^n}\right|\, \left|\frac{\partial f^n}{\partial z}  \right|\in \mathscr L^\infty_{\loc} (\Omega).
 \end{split}
 \]
 To see this latter inclusion we fix a compactly contained subdomain $\,G \Subset \Omega$. Recall that the functions  $\,g^n(z)\,$   and their c-uniform limit $\,g \,$ do not vanish. Thus we have a uniform bound from below
 $\abs{g^n (z)}\ge m \;$, for $\,z \in G\,$, $\,m\,$ being independent of $\,n = 1, 2, ... \,$. On the other hand the continuous functions $\, \frac{\partial f^n}{\partial z} (z) \;$ converge uniformly on $\,G\,$, so we also have the uniform bound $\, \left|\frac{\partial f^n}{\partial z} (z)\right| \le M\,$. This yields $\,\abs{A^n (z)} \le \frac{2k M |\alpha|}{m \,|\beta|} \,$ in $\,G\,$, as desired.

  Next we solve (uniquely) the nonhomogeneous Belrami equation
 \begin{equation}\label{heq4}
 \lambda^n_{\bar z} = \left[\,\mu^n (z) \lambda_z^n -A^n (z)\,\right] \chi_{_G}(z) \quad \textnormal{for a complex function }\lambda \in \mathscr W^{1,s} (\R^2)
 \end{equation}
 where $s>2$. We search the unknown function $\lambda^n (z)$ in the form of the Cauchy transform
of a complex density function  $\omega^n \in \mathscr L^s (\R^2)$, with $\,\supp \omega^n \subset G\,$, namely
\[\lambda^n (z)=  \frac{1}{\pi} \iint_{\C} \frac{\omega^n (\xi) \, \dtext \xi}{z-\xi} \, .\]
The density  function $\omega^n$  will be found (uniquely) by solving a singular integral equation
 \[\omega^n = \mu \, \chi_{_G} \, \mathcal S \omega^n -A^n \, \chi_{_G}\,, \qquad  \mathcal S \omega^n = - \frac{1}{\pi} \iint_{\C} \frac{\omega^n (\xi)}{(z-\xi)^2}\, \dtext \xi \, . \]
 Here   $ \mathcal S \colon \mathscr L^s (\C) \to \mathscr L^s (\C) $ is the familiar Beurling-Ahlfors operator. Thus
  \[\omega^n = - \left(I-\mu\,  \chi_{_G}\,  \mathcal S\right)^{-1}A^n \chi_{_G}\,\]
 We note that the operator \[ I-\mu\,  \chi_{_G}\,  \mathcal S \;\,\colon \mathscr L^s (\C) \to \mathscr L^s (\C)\]
is invertible for all  $\,2<s< 1+ \frac{1}{k}\,$, see~\cite{AIS}. We see that $\,\sup_n \norm{\omega^n}_{\mathscr L^s (\C)} < \infty$ and infer that the family $\{\lambda^n\}_{n=1}^\infty$ is equicontinuous. Indeed, we have a uniform H\"older  estimate,
\[\abs{\lambda^n (z_1) - \lambda^n (z_2)} \le C_s \, \abs{z_1-z_2}^{1-\frac{2}{s}} \norm{\omega^n}_{\mathscr L^s (\C)} . \]
We may, and do assume, by passing to a subsequence if necessary, that $\lambda^n (z)  \rightrightarrows \lambda (z)$ uniformly in $\C$. Therefore, the functions
\[e^{\lambda^n (z)} F^n (z) \rightrightarrows e^{\lambda (z)} F(z) \qquad \textnormal{uniformly on } G. \]
On the other hand the nowhere vanishing functions $H^n (z)=e^{\lambda^n} F^n$ satisfy a homogeneous Beltrami equation
\[H^n_{\bar z} = \mu (z) H^n_z,\;\quad\;\;\textnormal{in}\;\;G\]
which is straightforward from~\eqref{heq3} and~\eqref{heq4}.  In other words $H^n$ are $K$-quasiregular mappings with  $K= \frac{1+k}{1-k}$. Using Corollary~\ref{qrHur} we see that the limit function $\,H(z)=e^{\lambda (z)} F(z)\,$ does not vanish anywhere in $\,G\,$, unless $\,F\equiv 0\,$. This later case would mean that $\,F(z)=\alpha \, f(z) + \beta\, g(z) \equiv 0\,$ so, $J_\varphi (z)\equiv 0$, which we have already ruled out.
\end{proof}

\subsection{$p$-harmonic replacements}\label{secphar}
Here $\,\mathbb X\,$ and $\,\mathbb Y\,$ are circular domains.
Let $h\in \widetilde{\Ho}_p (\X, \Y)$, $p \ge 2$, and $\mathbb Q$ be an open square in $\Y_+$ such that $\mathbb Q \cap \Y$ is convex. Then the set
\[U=h^{-1} (\mathbb Q \cap \overline{\Y})=\{x \in \X \colon h(x) \in \mathbb Q \cap \overline{\Y}\}\]
is a simply connected domain in $\X$. We call it a {\it cell}. More specifically, $U$ will be called an {\it inner cell} if $\mathbb Q \subset \Y$ and a {\it boundary cell} if $\mathbb Q \cap \partial \Y \not= \varnothing\,$. Note that  $\mathbb Q$  can intersect only the outer boundary circle of $\Y$, since otherwise $\mathbb Q \cap \Y$ would not be convex.

\begin{proposition}\label{propha}
Let $h\in \widetilde{\Ho}_p (\X, \Y)$ and $U \subset \X$ be a cell. Then there exists $h^\ast \colon \overline{\X} \onto \overline{\Y}\,$, $h^\ast=h \colon \partial \X \onto \partial \Y$, such that	
\begin{enumerate}[(i)]
\item $h^\ast \in \widetilde{\Ho}_p (\X, \Y)$
\item $h^\ast = h \colon \overline{\X}\,  \setminus \, U \onto \overline{\Y} \setminus (\mathbb Q \cap \Y)$
\item $h^\ast \colon  U \onto \mathbb Q \cap \Y$ is a $p$-harmonic diffeomorphism
\end{enumerate}
Moreover,
\begin{enumerate}
\item[(iv)] $ \displaystyle \iint_{\X} \abs{\nabla h^\ast}^p \le \iint_{\X} \abs{\nabla h}^p$
\end{enumerate}
\end{proposition}
\begin{proof}
We begin with a general consideration,  without making explicit distinction between the two types of cells. Subsequently, additional details will be required for the case of a boundary cell. But at the end the proof will go in the same way for both cases. Recall that $\,h\,$ is a weak limit of homeomorphisms $\,h_j \in \Ho_p (\X, \Y)\,$.

We conveniently extend each $h_j$ and $h$, as in \S\ref{secext}, to obtain the monotone mappings $\widehat{h}_j \colon \widehat{\C} \onto \widehat{\C}$ and their limit $\,\widehat{h} \colon \widehat{\C} \onto \widehat{\C}\,$.  Note that $\{\widehat{h}_j\}^\infty_{j=1}$ are homeomorphisms in $\,\X\,$ as well as in  the reflected domains $\,\X_i\,$, $i=1,2, \dots , n\,$, but not necessarily  along the boundary circles, so the mappings
\[\widehat{h}_j \;\colon \;\X_+=\overline{\X} \cup \X_1 \cup \dots \cup \X_n \onto \overline{\Y} \cup \Y_1 \cup \dots \cup \Y_n =\Y_+\]
are only continuous and monotone. Let $\x_i$ and $ \Upsilon_i \, $, for some $i=1, \dots , n$, be the outer components of $\partial \X$ and $\partial Y$, respectively. Thus $\mathbb Q \subset \Y \cup {\Y}_i \cup \Upsilon_i$, which may or may not intersect $\Upsilon_i \, $. Recall from~\S\ref{secsqu} a sequence of open squares $\mathbb Q_n = \lambda_n \mathbb Q$, where  $0< \lambda_1 < \lambda_2 < \dots \to 1$, and the cells
\[\Omega_n = \widehat{h}^{-1}(\mathbb Q_n) \subset \X \cup \X_i \cup \x_i \, , \qquad \Omega_1 \Subset \Omega_2 \Subset \dots \Subset \dots \Omega=\widehat{h}^{-1}(\mathbb Q)\, .\]
Also recall the continua $C_n^j= \widehat{h}_j^{-1} (\partial \mathbb Q_n)$, the bounded components $\Omega_n^j=\widehat{h}_j^{-1}(\mathbb Q_n)$ of $\widehat{\C} \setminus C_n^j\,$ and the inclusions
\[\overline{\Omega}_{n-1} \subset \Omega_n^j \subset \overline{\Omega}_n^j \subset \Omega_{n+1}\, , \qquad n \ge 2 \textnormal{ for } j \ge j_n \, . \]
In case of the inner cell $U=h^{-1} (\mathbb Q)$, the square $\mathbb Q$ lies entirely in $\Y$, so the continua $C_n^j$ are  Jordan curves, because $h_j=\widehat{h}_j \colon \Omega_n^j \onto  \mathbb Q_n$ are homeomorphisms. Let us fix $j$, say $j=j_n$, and introduce a sequence of homeomorphisms $h_{j_n} \colon \Omega_n^{j_n} \onto \mathbb Q_n$, $n=2,3, \dots$ We shall return to these homeomorphisms after thorough discussion of a similar construction for the boundary cells. Let
\[U=h^{-1} (\mathbb Q \cap \overline{\Y})=\{x \in \X \colon h(x) \in \mathbb Q \cap \overline{\Y}\} = \Omega \cap \X \, .   \]
As observed in \S\ref{secspl} this simply connected domain has resulted by cutting $\Omega=\widehat{h}^{-1} (\mathbb Q)$ along an arc in the boundary circle $\x = \x_i$. This arc splits each of the cells $\Omega_1, \Omega_2, \dots\;$  into the pair of two simply connected subdomains; the one which lies in $\X$ will remain of interest to us. We denote it by $U^n = \Omega_n \cap \X$. Clearly, $\bigcup_{n=1}^\infty U^n =U$. On the other hand, for  $\,j \ge j_n\,$ we have simply connected domains $\Omega_n^j = \widehat{h}_j^{-1} (\mathbb Q_n)$ which are the inner complements of the continua $C_n^j = \widehat{h}_j^{-1} (\partial \mathbb Q_n) \subset \Omega_{n+1}$. But we will be interested only in a  portion of $\Omega_n^j\,$ that is compactly contained in $\X$.  To this effect, we shear  the set $\,\mathbb Q \cap \overline{\Y}\,$  along the circular part of its boundary. Precisely, for $\delta >0$ sufficiently small we consider the region $\mathbb Q_n \cap \Y^\delta \subset \mathbb Q_n \cap \Y$, where $\Y^\delta = \{y \in \Y \colon \dist (y, \partial \Y)> \delta\}$. The portion of $\Omega_n^j$ that we are going to consider is a Jordan domain $\Omega_n^{j, \,\delta}=h_j^{-1} (\mathbb Q_n \cap \Y^\delta) \subset \X\,$ with $j = j_n$. Here  $\,\delta = \delta_n >0$ is chosen as follows. For $\delta=0$ we have
\[\Omega_n^{j,\,0}=h_j^{-1}(\mathbb Q_n \cap \Y)= \Omega_n^j \cap \X \,\;,\;\;\;\; j = j_n . \]
Since $h_{j} \colon \X \onto \Y\,$ is a homeomorphism, there is $\delta = \delta_n>0$ such that the set  $\Omega_n^{j, \,\delta}= h_j^{-1}(\mathbb Q_n \cap \Y^\delta)$  contains $\Omega_n^j \cap \X^{\epsilon_n}$, where  $\,\X^{\epsilon_n}=  \{x \in \X \colon \dist (x, \partial \X)> \epsilon _n \}$, say  $\,\varepsilon_n =\frac{1}{n}\,$. Let us write it as
\[ \Omega_n^{j, \,\delta} =h_j^{-1}(\mathbb Q_n \cap \Y^\delta)\supset \overline{\Omega}_{n-1} \cap \X^{\epsilon_n} \, , \quad \textnormal{ where } j=j_n\, , \; \delta = \delta_n \textnormal{ and } \epsilon_n = \frac{1}{n}\, .\]
Thus we have a homeomorphism
\begin{equation}\label{eq98}
h_{j_n} \colon {\Omega}_n^{j_n, \,\delta_n} \onto \mathbb Q_n \cap \Y^{\delta_n}
\end{equation}
of a Jordan domain ${\Omega}_n^{j_n,\,\delta_n}$ onto the convex domain $\mathbb Q_n \cap \Y^{\delta_n}$.\\
From now on the remaining part of the proof treats both types of cells in the same way. Simply, in case of the inner cells we have ${\Omega}_n^{j_n, \, \delta_n}=\Omega_n^{j_n}$ and $\mathbb Q_n \cap \Y^{\delta_n}= \mathbb Q_n$.  We  now appeal to Theorem~\ref{AS}. Accordingly,
each homeomorphism in~\eqref{eq98}
can be replaced by a $p$-harmonic diffeomorphism, denoted by
\[\widetilde{h}_{j_n}= \begin{cases} h_{j_n} \colon \X \setminus {\Omega}_n^{j_n, \,\delta_n} \to \Y \setminus (\mathbb Q_n \cap \Y^{\delta_n}) \\
\textnormal{$p$-harmonic diffeomorphism of } {\Omega}_n^{j_n,\, \delta_n} \onto \mathbb Q_n \cap \Y^{\delta_n} . \end{cases}\]
By Theorem \ref{winlem} we still have $\widetilde{h}_{j_n} \in \Ho_p (\X, \Y)\,$ and
\[\mathcal E_\X [\widetilde{h}_{j_n}] \le \mathcal E_\X [{h}_{j_n}] \, , \qquad \textnormal{for } j_1<j_2 < \dots \to \infty \, .\]
 Now the desired  mapping $h^\ast \in \widetilde{\Ho}_p (\X, \Y) \,$ is defined to be the weak limit of $\widetilde{h}_{j_n}$.  Recall that  homeomorphisms in $\widetilde{\Ho}_p (\X, \Y) $ admit a continuous extension to the closure of $\,\X\,$, with values in  $\,\overline{\Y}\,$.  Furthermore, $\,h_{j_n}\,$ converges uniformly to $\,h^\ast\,$ on $\,\overline{\X}\,$, because of a uniform bound of energy and Lemma \ref{lem212}. For $x \in \overline{\X} \setminus U$ we have $\widetilde{h}_{j_n} = h_{j_n} \to h(x)$, so $h^\ast = h$ on $\overline{\X} \setminus U$. Moreover, $h^\ast$ is a $p$-harmonic mapping on $U=\Omega \cap \X$. By Theorem~\ref{winlem} we see that $h^\ast \in h + \W^{1,p}_\circ (U)$ and
\[\mathcal E_\X [h^\ast] \le \mathcal E_\X [h].\]
It remains to show that $h^\ast \colon U \cap \X \onto  \mathbb Q \cap \Y$ is a diffeomorphism. To this effect, let $U'$ be any  compactly contained subdomain of $U$. We note that $U' \subset \Omega_n^{j_n, \delta_n}$  for sufficiently large $n$. By Hurwitz type Theorem~\ref{huthm1} $h^\ast \colon U' \to \C$ is either a diffeomorphism or its Jacobian determinant vanishes identically on $U'$. However, this latter case is easily ruled out by computing the integral of the Jacobian over $U$.  For this purpose we appeal to the equation~\eqref{averJac},
\[
\begin{split}
\int_{U} J(x, h^\ast)\, \dtext x  & = \lim\limits_{n \to \infty}  \int_{U} J(x, \widetilde{h}_{j_n})\, \dtext x  \ge \lim\limits_{n \to \infty}  \int_{{\Omega}^{j_n, \delta_n}_n}  J(x, \widetilde{h}_{j_n})\, \dtext x \\ &=   \lim\limits_{n \to \infty} \abs{\mathbb Q_n \cap \Y^{\delta_n}} = \abs{\mathbb Q \cap \Y}>0\,.
\end{split}
\]
In particular, $\int_{U'} J(x, {h^\ast}) \, \dtext x>0$ for sufficiently large subdomains $U' \Subset U$. In conclusion, $h^\ast \colon U' \to \C$ is a diffeomorphism. Letting $U'$ approach  $U$ we conclude that $h^\ast \colon U \to \C$ is a diffeomorphism. Recall that homeomorphisms $\widetilde{h}_{j_n}$ take $\Omega_n^{j_n , \delta_n} \subset U$ onto $\mathbb Q_n \cap \Y^{\delta_n} \subset \mathbb Q_n \cap \Y$ and converge uniformly on $U$ to $h^\ast$. Hence, by elementary topological arguments we see that $h^\ast \colon U \onto \mathbb Q \cap \Y$ is a diffeomorphism.

\subsection{Further replacements} A key step in constructing homeomorphisms $\, h_j^*  : \mathbb X \onto \mathbb Y \,$ for Theorem \ref{thmain} is  the following
\begin{proposition}\label{keylem}
Let $\X$ and $\Y$ be circular domains bounded by circles $\x_1 , \dots , \x_\ell$ and $\Upsilon_1, \dots , \Upsilon_\ell$, respectively, and $h \in \widetilde{\Ho}_p (\X, \Y)$. Select a pair of boundary circles $\x=\x_i$, $\Upsilon = \Upsilon_i\,$ and recall that $\,h \colon \x_i \onto \Upsilon_i\,$.  Consider the open region in $\X$
\[\Omega = h^{-1} (\Y \cup \Upsilon) = \{x \in \X \colon h(x) \in \Y \cup \Upsilon\}\, .\]
Then for every $\epsilon >0$ there exists $\,H \colon \X \to \overline{\Y}\setminus\Upsilon\,$\,in\, $\,\widetilde{\Ho}_p (\X, \Y)\,$ which satisfies the following conditions:
\begin{enumerate}[(a)]
\item\label{con1} $H=h \colon \partial \X \onto \partial \Y\,$, also $H=h \colon \X \setminus \Omega \to \partial \Y \setminus \Upsilon$
\item\label{con2} $H \colon \Omega \onto \Y$ is a homeomorphism
\item\label{con3} $\norm{H-h}_{\mathscr R^p (\X)} \le \epsilon\;$
\end{enumerate}
\end{proposition}

Of course $H$ depends on $\epsilon$. However, to ease the writing, we suppress the explicit dependence on $\,\epsilon\,$ in the notation of $\,H\,$ but return to it later on.

\begin{proof}
We shall make use of  Proposition~\ref{propha}, which requires that $\,\Upsilon\,$ be the outer boundary of $\,\Y\,$. Let us first demonstrate how Proposition~\ref{keylem} can be reduced to this case.

{\bf Step 1} {\it (Reduction to the case of outer boundary)}  Given a pair $\,(\mathfrak X_i , \Upsilon_i)\,$ of boundary components, we recall that $\;h \colon \x_i \onto \Upsilon_i\;$. We explore the reflections  $\;\varphi_i \colon \X \onto \X_i\;$ and $\;\psi_i \colon \Y \onto \Y_i\;$ from \S\ref{secext}. In this way $\,\x_i\,$ and $\,\Upsilon_i\,$ become outer boundaries of the circular domains $\,\X_i\,$ and $\,\Y_i\,$, respectively. Consider the mapping
\[h^i = \psi_i \circ h \circ \varphi_i \colon \X_i \to \Y_i \, , \qquad h^i \in \widetilde{\Ho}_p (\X_i, \Y_i)\, .\]
Assuming that Proposition~\ref{keylem} holds in case of outer boundaries we infer that given any $\epsilon^i >0$ (to be revealed later on) there exists $\,H^i \colon \X_i \to \overline{\Y}_i \setminus \Upsilon_i$, $\,H^i \in \widetilde{\Ho}_p (\X_i, \Y_i)$ such that
\begin{itemize}
\item  $H^i=h^i \colon \overline{\X}_i \setminus \Omega' \to \partial \Y_i \setminus \Upsilon_i\;$, where $\Omega^i = \{x \in \X_i \colon h^i(x) \in \Y_i \cup \Upsilon_i\}$
\item $H^i \colon \Omega^i \onto \Y_i$ is a homeomorphism
\item $\norm{H^i-h^i}_{\mathscr R^p (\X_i)} \le \epsilon^i$
\end{itemize}
We then find the desired mapping $H\in \widetilde{\Ho}_p (\X, \Y)$ by setting
\[H= \psi_i \circ H^i \circ \varphi_i \colon \X \to \Y\]
 Indeed, the conditions~\eqref{con1} and~\eqref{con2} are obviously satisfied. Regarding the condition~\eqref{con3} we note that $\psi_i$ and $\varphi_i$ are $\mathscr C^\infty$-smooth diffeomorphism so
\[\norm{H-h}_{\mathscr R^p (\X)} = \norm{\psi_i \circ H^i \circ \varphi_i -\psi_i \circ h^i \circ \varphi_i }_{\mathscr R^p (\X)} \le C \, \norm{H^i-h^i}_{\mathscr R^p (\X_i)} \le C \, \epsilon^i = \epsilon  \]
where $C$ depends only on the Lipschitz constants of the reflections; thus we take  $\,\epsilon^i\,$  equal to $\,\epsilon /C$. The interested reader may wish to observe that this transition (by reflecting about circles) retains control over $\,\norm{H-h}_{\mathscr R^p (\X)}$, but  at the sacrifice of loosing the sharp inequality $\,\norm{H}_{\mathscr R^p (\X)}\,\leqslant \norm{h}_{\mathscr R^p (\X)}$. However, we shall have no need of this sharp inequality.

{\bf Step 2} {\it (Three families of squares)}
\begin{lemma}\label{lemuni}
Let $\mathbb F$ be a  compact subset in $\C$ and $\rho >0$. There exist three families $\mathcal A = \{A_\alpha\}_{\alpha =1}^\infty$,  $\mathcal B = \{B_\beta\}_{\beta =1}^\infty$ and  $\mathcal C = \{C_\gamma\}_{\gamma =1}^\infty$, each of them consists of disjoint open squares in $\C \setminus \mathbb F$ of diameter $\rho$ or less, such that
\[\bigcup_{\alpha =1}^\infty A_\alpha \cup \bigcup_{\beta =1}^\infty B_\beta \cup \bigcup_{\gamma =1}^\infty C_\gamma = \C \setminus \mathbb F \, .\]
\end{lemma}
The construction of such families presents no difficulty. We shall not bother the  reader with an explicit construction of $\mathcal A , \mathcal B $ and $\mathcal C$.
In our application the set $\mathbb F\,$ will be the union of all bounded components of $\C \setminus \Y$. We will be dealing only with those open squares $\mathbb Q$ which intersect $\Y$, so  $\mathbb Q \cap \Y\,$ will be  nonempty and
\begin{equation}\label{eq51}
\bigcup_{\alpha =1}^\infty A_\alpha \cup \bigcup_{\beta =1}^\infty B_\beta \cup \bigcup_{\gamma =1}^\infty C_\gamma\supset \Y \cup \Upsilon \, .
\end{equation}\\
To each family $\mathcal A , \mathcal B $ and $\mathcal C$  there will correspond a mesh of cells in $\,\mathbb X\,$. We begin with the family $\mathcal A \,$.

{\bf Step 3} {\it ($p$-harmonic replacement in $\,\mathcal A$-cells)} Fix $\rho >0$ for the moment, small enough so that if an open square $\mathbb Q$ of diameter $\rho$ or less intersects the outer boundary $\Upsilon = \Upsilon_i \subset \partial \Y$ then it lies entirely in $\Y \cup \Upsilon_i \cup \Y_i$. Hence, in particular, $\mathbb Q \cap \Y$ is convex.
 Consider the mesh of open cells in $\X$
\[U_\alpha = h^{-1} (A_\alpha \cap \overline{\Y})=\{x \in \X \colon h(x) \in A_\alpha \cap \overline{\Y}\} \, . \]
It should be noted that $U_\alpha$ is empty whenever $A_\alpha \cap \Y = \emptyset$.

\begin{center}\begin{figure}[h]
\includegraphics[width=0.9\textwidth]{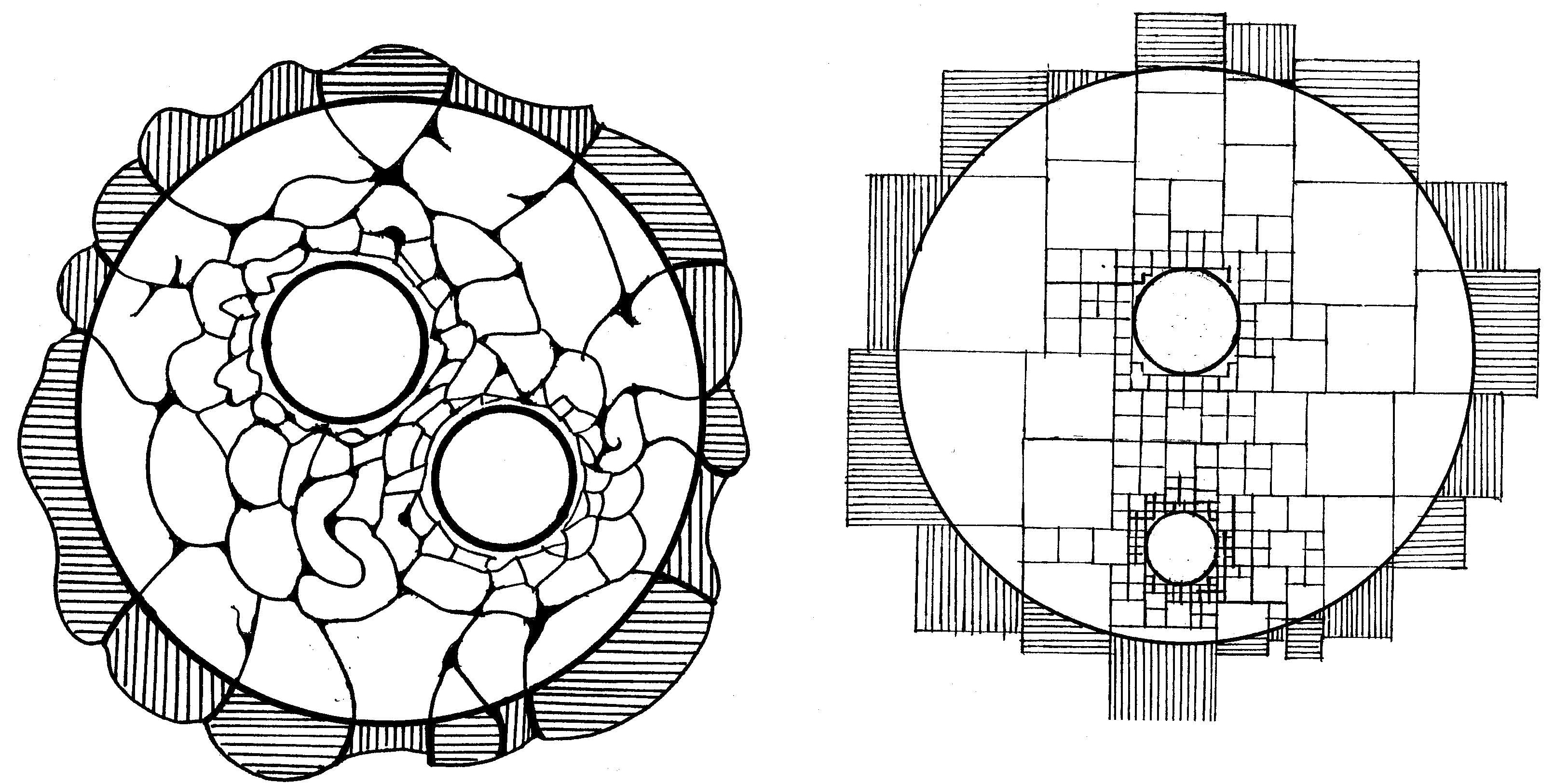} \caption{A mesh of cells in $\X$ and squares in $\Y$}\label{fig1}
\end{figure}\end{center}

We now appeal to Proposition~\ref{propha} which provides us with a mapping that we denote by  $\,h_{\mathcal A} \colon \overline{\X} \onto \overline{\Y}\,$,  such that
\begin{enumerate}[(i)]
\item\label{cond0} $\displaystyle  h_{\mathcal A} \in \widetilde{\Ho}_p (\X, \Y) \, , \quad h_{\mathcal A}=h \colon \partial \X \onto \partial \Y$
\item\label{cond1} $\displaystyle h_{\mathcal A}=h \colon \overline{\X} \setminus \bigcup_{\alpha =1}^\infty U_\alpha \onto \overline{\Y} \setminus \bigcup_{\alpha=1}^\infty (A_\alpha \cap \Y)$
\item\label{cond2} $\displaystyle h_{\mathcal A} \colon \bigcup_{\alpha=1}^\infty U_\alpha \onto \bigcup_{\alpha=1}^\infty (A_\alpha \cap \Y)$ is a $p$-harmonic diffeomorphism.
\item\label{cond3} $\displaystyle \iint_{\X} \abs{\nabla h_{\mathcal A}}^p \le \iint_{\X} \abs{\nabla h}^p$.
\end{enumerate}
The actual construction of $h_{\mathcal A}$ is accomplished step by step. First, using Proposition~\ref{propha} for $h$ and $U=U_1$, we obtain what has been denoted by $h^\ast$. In the second step we use Proposition~\ref{propha} for $h^\ast$ in place of $h$ and for $U=U_2$ to obtain $h^{\ast \ast}$. This process continues indefinitely resulting in a sequence of mappings converging weakly to $\,h_{\mathcal A}$.  Conditions~\eqref{cond1}--\eqref{cond3} are clearly satisfied. However, the condition~\eqref{cond0} calls for some explanation. Upon every finite number of steps we obtain a mapping $h^{\ast \ast \dots \ast} \in \widetilde{\Ho}_p (\X, \Y)$. These are weak $\W^{1,p}$-limits of homeomorphism in $\Ho_p (\X, \Y)$. On the other hand, as the number of steps approaches $\infty$, the mappings $h^{\ast \ast \dots \ast}$ converge weakly in $\W^{1,p} (\X, \Y)$ to $h_{\mathcal A}$. Hence $h_{\mathcal A} \in  \widetilde{\Ho}_p (\X, \Y)\,$, as desired.

 It should be recalled at this point that the family $\mathcal A$ of squares actually depends on $\rho >0\,$. We write it as $\mathcal A= \mathcal A_\rho$. Therefore, for each $\rho$ we have mappings  $\,h_{\mathcal A_\rho} \in \widetilde{\Ho}_p (\X, \Y)\,$ which satisfy conditions~\eqref{cond1}--~\eqref{cond3}. Let us now take a close look at the mappings $h_{\mathcal A_\rho}$ when $\rho$ approaches $0$. First, given $x\in \X$, we see that  $h_{\mathcal A} (x)=h(x)$ if $x$ does not belong to any of the cells $U_\alpha$. If, however, $x\in U_\alpha$ for some $\alpha$ then
\[h_{\mathcal A}(x) \in h (U_\alpha) = A_\alpha \cap \overline{\Y}\, .\]
In either case $\abs{h_\alpha (x)-h(x)} \le \rho$.
We then infer that $\norm{h_{\mathcal A_\rho}-h}_{\mathscr C (\X)} \le \rho$,
which means that $h_{\mathcal A_\delta} \rightrightarrows h$ uniformly on $\X$ as $\rho \to 0$.
On the other hand the energy of $h_{\mathcal A_\delta}$ is not greater than that of $h$. Thus $h_{\mathcal A_\delta} \to h$, weakly in $\W^{1,p}(\X)$. Hence
\[\mathcal E_\X [h] \le \lim_{\rho \to \,0} \mathcal E_{\X} [h_{\mathcal A_\rho}] \le \lim_{\rho \to 0} \mathcal E_{\X} [h] =\mathcal E_{\X} [h] \, .     \]
This shows that in fact $\lim\limits_{\rho \to \,0} \mathcal E_\X [h_{\mathcal A_\rho}]  = \mathcal E_\X [h]$. Consequently, $h_{\mathcal A_\delta} \to h $ strongly in $\mathscr R^p (\X)$.
In summary, given any $\epsilon >0$, we can choose and fix $\,\rho > 0 \,$ for the family $\,\mathcal A= \{A_\alpha\}_{\alpha=1}^\infty\,$  to satisfy
\begin{equation}\label{eq61}
\norm{h_{\mathcal A}-h}_{\mathscr R^p (\X)} \le \frac{1}{3}\, \epsilon
\end{equation}

{\bf Step 4} {\it ($p$-harmonic replacements in $\,\mathcal B$-cells)} We repeat the procedure in Step 3, with $h_{\mathcal A}$ in place of $h$ and with  the family  $\mathcal B$ in place of $\mathcal A$; with the same $\epsilon >0$. Thus, we consider the mesh of open cells in $\,\X\,$  \[\,V_\beta = h^{-1}_{\mathcal A} (B_\beta \cap \overline{\Y})= \{x\in \X \colon h_{\mathcal A} (x) \in B_\beta \cap \overline{\Y}\}\, ,\] and we obtain a mapping, denoted by $h_{\mathcal A \, \mathcal B} \colon \overline{\X} \onto \overline{\Y}$, such that
\begin{enumerate}[(i)]
\item\label{cond00} $\displaystyle  h_{\mathcal A \, \mathcal B} \in \widetilde{\Ho}_p (\X, \Y), \quad h_{\mathcal A \, \mathcal B}= h_{\mathcal A}=h \colon \partial \X \onto \partial \Y$
\item\label{cond11} $\displaystyle h_{\mathcal A \, \mathcal B}=h_{\mathcal A} \colon \overline{\X} \setminus \bigcup_{\beta =1}^\infty V_\beta \onto \overline{\Y} \setminus \bigcup_{\beta=1}^\infty (B_\beta \cap \Y)$
\item\label{cond22} $\displaystyle h_{\mathcal A \, \mathcal B} \colon \bigcup_{\beta=1}^\infty V_\beta \onto \bigcup_{\beta=1}^\infty (B_\beta \cap \Y)$ is a $p$-harmonic diffeomorphism.
\item\label{cond33} $\displaystyle \iint_{\X} \abs{\nabla h_{\mathcal A \, \mathcal B}}^p \le \iint_{\X} \abs{\nabla h_{\mathcal A}}^p \le \iint_{\X} \abs{\nabla h}^p$
\end{enumerate}
The inequality parallel to~\eqref{eq61} reads as
\[ \norm{h_{\mathcal A \, \mathcal B} - h_{\mathcal A}   }_{\mathscr R^p (\X)} \le \frac{1}{3} \, \epsilon \, .\]

{\bf Step 5} {\it ($p$-harmonic replacements in $\,\mathcal C$-cells)} The above procedure of $p$-harmonic replacements is finally  carried over to the mappings $h_{\mathcal A \, \mathcal B}$. We consider the mesh of open cells in $X$,
\[W_\gamma = h^{-1}_{\mathcal A \, \mathcal B} (C_\gamma \cap \overline{\Y})= \{x\in \X \colon h_{\mathcal A \, \mathcal B} (x) \in C_\gamma \cap \overline{\Y}\} \, .\]
The replacements result in the mapping, denoted by $h_{\mathcal A \, \mathcal B \, \mathcal C} \colon \overline{\X} \onto \overline{\Y}$, such that
\begin{enumerate}[(i)]
\item\label{cond000} $\displaystyle  h_{\mathcal A \, \mathcal B \, \mathcal C} \in \widetilde{\Ho}_p (\X, \Y)\, , \quad h_{\mathcal A \,  \mathcal B \, \mathcal C}= h_{\mathcal A\, \mathcal B}=h_{\mathcal A}=h  \colon \partial \X \onto \partial \Y$
\item\label{cond111} $\displaystyle h_{\mathcal A \, \mathcal B \, \mathcal C}=h_{\mathcal A \, \mathcal B} \colon \overline{\X} \setminus \bigcup_{\gamma =1}^\infty W_\gamma \onto \overline{\Y} \setminus \bigcup_{\beta=1}^\infty (C_\gamma \cap \Y)$
\item\label{cond222} $\displaystyle h_{\mathcal A \, \mathcal B \, \mathcal C} \colon \bigcup_{\gamma=1}^\infty W_\gamma \onto \bigcup_{\beta=1}^\infty (C_\gamma \cap \Y)$ is a $p$-harmonic diffeomorphism.
\item\label{cond333} $\displaystyle \iint_{\X} \abs{\nabla h_{\mathcal A \, \mathcal B \, \mathcal C}}^p \le \iint_{\X} \abs{\nabla h_{\mathcal A \, \mathcal B }}^p \le \iint_{\X} \abs{\nabla h}^p$
\end{enumerate}
Moreover,
\begin{equation}
\norm{h_{\mathcal A \, \mathcal B \, \mathcal C} - h_{\mathcal A \, \mathcal B}   }_{\mathscr R^p (\X)} \le \frac{1}{3} \, \epsilon \,.
\end{equation}

{\bf Step 6} {\it (The desired mapping $\,H=h_{\mathcal A \, \mathcal B \, \mathcal C}$)} Upon three consecutive replacement procedures in the $\mathcal A$-cells, $\mathcal B$-cells and $\mathcal C$-cells we finally arrive at the mapping $h_{\mathcal A \, \mathcal B \, \mathcal C}$. This map turns out to satisfy all the conditions asserted  in Proposition~\ref{keylem}, so we denote it by
\[H=h_{\mathcal A \, \mathcal B \, \mathcal C} \in \widetilde{\Ho}_p (\X, \Y)\, .\]
Let us verify those conditions. Condition~\eqref{con3} is obvious;
\[ \begin{split}
\norm{H-h}_{\mathscr R^p (\X)} &\le \norm{h_{\mathcal A \, \mathcal B \, \mathcal C}-h_{\mathcal A \, \mathcal B}}_{\mathscr R^p (\X)} + \norm{h_{\mathcal A \, \mathcal B}-h_{\mathcal A}}_{\mathscr R^p (\X)} + \norm{h_{\mathcal A }-h}_{\mathscr R^p (\X)}\\ &\le \frac{\epsilon }{3} + \frac{\epsilon }{3} + \frac{\epsilon }{3} = \epsilon \, .
\end{split}
\]
Verification of other conditions involves elementary set theoretical considerations.  Let use include them for completeness; some do not seem immediate at all. Let us begin with the following identity
\begin{equation}\label{equnion}
\bigcup_{\alpha =1}^\infty U_\alpha \cup \bigcup_{\beta =1}^\infty V_\beta \cup \bigcup_{\gamma =1}^\infty W_\gamma = \Omega
\end{equation}
 where $\Omega = \{x\in \X \colon h(x) \in \Y \cup \Upsilon\}$.
 \begin{proof}[Proof of~\eqref{equnion}]
Take any point in $\X$ such that
\[x \not \in \bigcup_{\alpha =1}^\infty U_\alpha \cup \bigcup_{\beta =1}^\infty V_\beta \cup \bigcup_{\gamma =1}^\infty W_\gamma \, .  \]
Since $x\not \in \bigcup_{\alpha =1}^\infty U_\alpha $, we see that $h(x)=h_{\mathcal A} (x)$. Since $x\not \in  \bigcup_{\beta =1}^\infty V_\beta$, we see that $h_{\mathcal A}(x)= h_{\mathcal A \, \mathcal B}(x)$. Since $x \not \in \bigcup_{\gamma =1}^\infty W_\gamma$, we see that $h_{\mathcal A \, \mathcal B}(x)= h_{\mathcal A \, \mathcal B \, \mathcal C}(x) = H(x)$. Hence $y:= h(x)= h_{\mathcal A}(x)= h_{\mathcal A \, \mathcal B}(x)=  h_{\mathcal A \, \mathcal B \, \mathcal C}(x)$. Moreover, we have
\[y=h(x) \not \in \bigcup_{\alpha =1}^\infty (A_\alpha \cap \overline{\Y})\, , \quad  y=h_{\mathcal A}(x) \not \in \bigcup_{\beta =1}^\infty (B_\beta \cap \overline{\Y}) \quad \textnormal{and} \quad y=h_{\mathcal A \, \mathcal B}(x) \not \in \bigcup_{\gamma =1}^\infty (C_\gamma \cap \overline{\Y}) \, . \]
Thus
\[
y \not \in \left(\bigcup_{\alpha =1}^\infty A_\alpha \cup \bigcup_{\beta =1}^\infty B_\beta \cup \bigcup_{\gamma =1}^\infty C_\gamma  \right) \cap \overline{\Y} \, .
\]
On the other hand, by the formula~\eqref{eq51} \,, $\,h(x) \not \in \Y \cup \Upsilon$ so $x \not \in \Omega\,$. \\

Conversely, suppose $x\in \bigcup_{\alpha =1}^\infty U_\alpha \cup \bigcup_{\beta =1}^\infty V_\beta \cup \bigcup_{\gamma =1}^\infty W_\gamma \,$. Then we have three cases

{\bf Case 1} $x\in \bigcup_{\alpha =1}^\infty U_\alpha$, which yields $h(x) \in A_\alpha \cap \overline{\Y} \subset \Y \cup \Upsilon$, so $x\in \Omega$.

{\bf Case 2} $x\not \in \bigcup_{\alpha =1}^\infty U_\alpha$ and $x\in \bigcup_{\beta =1}^\infty V_\beta$. This yields \[h(x)= h_{\mathcal A} (x) \in B_\beta \cap \overline{\Y} \subset \Y \cup \Upsilon \, , \quad x\in \Omega . \]

{\bf Case 3} $x\not \in \bigcup_{\alpha =1}^\infty U_\alpha\,$ and  $\,x\not \in \bigcup_{\beta =1}^\infty V_\beta\,$, so $\,x\in \bigcup_{\gamma =1}^\infty W_\gamma$. This yields
\[h(x)= h_{\mathcal A} (x) =  h_{\mathcal A \, \mathcal B} (x)  \in C_\gamma \cap \overline{\Y} \subset \Y \cup \Upsilon \, , \quad x\in \Omega . \]
 completing the proof of (\ref{equnion}).
\end{proof}
A care needs to be excercised in showing that  $H \colon \Omega \to \Y\,$ is injective. For this, consider two distinct points
\[x_1 , x_2 \in \Omega=  \bigcup_{\alpha =1}^\infty U_\alpha \cup \bigcup_{\beta =1}^\infty V_\beta \cup \bigcup_{\gamma =1}^\infty W_\gamma\]
and suppose, to the contrary, that $y:= H(x_1)=H(x_2)$. Since the mapping $H=h_{\mathcal A \, \mathcal B \, \mathcal C} \colon  \bigcup_{\gamma =1}^\infty W_\gamma  \onto  \bigcup_{\gamma =1}^\infty (C_\gamma \cap \Y )\;$
is injective and
\[h_{\mathcal A \, \mathcal B \, \mathcal C}  = h_{\mathcal A \, \mathcal B} \colon \X \setminus  \bigcup_{\gamma =1}^\infty W_\gamma  \to \overline{\Y} \setminus \bigcup_{\gamma =1}^\infty (C_\gamma \cap \Y )  \]
it follows that $x_1 , x_2 \not \in \bigcup_{\gamma =1}^\infty W_\gamma $. Thus we are reduced to the case $x_1, x_2 \in \bigcup_{\beta =1}^\infty U_\alpha \cup \bigcup_{\beta =1}^\infty V_\beta$ and $y= h_{\mathcal A \, \mathcal B} (x_1) = h_{\mathcal A \, \mathcal B} (x_2)$. In exactly the same way we infer that $x_1, x_2 \not \in  \bigcup_{\beta =1}^\infty V_\beta $. Thus we are further reduced to the case
\[x_1, x_2 \in  \bigcup_{\alpha =1}^\infty U_\alpha \quad \textnormal{ and } \quad y=h_{\mathcal A} (x_1) = h_{\mathcal A} (x_2)\, .\]
Just as before $x_1, x_2 \not \in  \bigcup_{\alpha =1}^\infty U_\alpha $. We arrive at the contradiction
\[x_1 , x_2 \not \in   \bigcup_{\alpha =1}^\infty U_\alpha \cup \bigcup_{\beta =1}^\infty V_\beta \cup \bigcup_{\gamma =1}^\infty W_\gamma \]
finishing the argument for injectivity.

To see that $H \colon \Omega \into \Y$, we argue in similar way; namely, let
\[x\in  \Omega= \bigcup_{\alpha =1}^\infty U_\alpha \cup \bigcup_{\beta =1}^\infty V_\beta \cup \bigcup_{\gamma =1}^\infty W_\gamma \, .\]
If $x\in \bigcup_{\gamma =1}^\infty W_\gamma$, then
\[y:= H(x)=  h_{\mathcal A \, \mathcal B \, \mathcal C} \in  \bigcup_{\gamma =1}^\infty C_\gamma \cap \Y \subset \Y \, .\]
Thus suppose $x \not \in \bigcup_{\gamma =1}^\infty W_\gamma$, so $H(x)= h_{\mathcal A \, \mathcal B}(x)$ and $x\in \bigcup_{\alpha =1}^\infty U_\alpha \cup \bigcup_{\beta =1}^\infty V_\beta  $. If $x\in  \bigcup_{\beta =1}^\infty V_\beta$, then $y= h_{\mathcal A \, \mathcal B} \in  \bigcup_{\beta =1}^\infty B_\beta \cap \Y \subset \Y$. Then suppose that $x \not \in  \bigcup_{\beta =1}^\infty V_\beta$, so $x\in  \bigcup_{\beta =1}^\infty U_\alpha$. This yields $y=h_{\mathcal A}(x) \in  \bigcup_{\alpha =1}^\infty A_\alpha \cap \Y \subset \Y$ which completes the proof of the inclusion $H(\Omega) \subset \Y$.

 Lastly, to see surjectivity $\,H \colon \Omega \onto \Y\,$ we recall that  $\,H=h_{\mathcal A \, \mathcal B \, \mathcal C} \colon \overline{\X} \onto \overline{\Y}$ and $H = h \colon \partial \X \onto \partial \Y$. Hence $H(\X) \supset \Y$. Take any  $y\in \Y$. There exists $x\in \X$ such that $H(x)=y$. Suppose  that $x\not \in \Omega=  \bigcup_{\alpha =1}^\infty U_\alpha \cup \bigcup_{\beta =1}^\infty V_\beta \cup \bigcup_{\gamma =1}^\infty W_\gamma$. This leads to a clear contradiction,
\[y=H(x)= h_{\mathcal A \, \mathcal B \, \mathcal C}  = h_{\mathcal A \, \mathcal B} = h_{\mathcal A} =h(x)\in \Y\, , \quad \textnormal{ and } \quad x \in h^{-1} (\Y \cup \Upsilon) = \Omega \, .\]
The proof of Proposition~\ref{keylem} is complete.
\end{proof}

\subsection{Proof of Theorem~\ref{thmain} for multiple connected domains}
Given $h \in \widetilde{\Ho}_p (\X, \Y)\,$, we define by induction a chain of mappings $H_0, H_1, \dots , H_\ell \in \widetilde{\Ho}_p (\X, \Y)\,$, setting  $\,H_0 \equiv h\,$. Suppose that $H_{k-1} \in \widetilde{\Ho}_p (\X, \Y)$ is already defined for some $1 \le k \le \ell$. To define $H_k$ we consider the boundary circles $\x_k \subset \partial \X$ and $\Upsilon_k \subset \partial \Y$, and recall that $h \colon \x_k \onto \Upsilon_k$. We consider the open regions
\[\Omega_k = H^{-1}_{k-1} (\Y \cup \Upsilon_k)= \{x \in \X \colon H_{k-1}(x) \in \Y \cup \Upsilon_k\}\, .\]
By Propposition~\ref{keylem} applied to $H_{k-1}$ in place of $h$ we obtain a mapping $\,H\,$, which for our purpose will be denoted by  $H_{k} \in \widetilde{\Ho}_p (\X, \Y)$. We have
\begin{enumerate}[(i)]
\item $H_k=H_{k-1} \colon \overline{\X} \setminus \Omega_k \to \partial \Y \setminus \Upsilon_k\, , \quad H_k =H_{k-1} \colon \partial \X \onto \partial \Y$
\item $H_k \colon \Omega_k \onto \Y$ is a homeomorphism
\item $\norm{H_k-H_{k-1}}_{\mathscr R^p (\X)} \le \epsilon$
\end{enumerate}
The desired sequence  $h^\ast_j \in \widetilde{\Ho}_p (\X, \Y)$ of homeomorphisms $h_j^\ast \colon \X \onto \Y\,$  converging to  $\,h\,$ strongly in $\W^{1,p} (\X, \Y)\,$ will be defined as follows. Let a sequence of positive numbers $\epsilon_j \to 0\,$ be chosen and fixed. Then we set  $\,h_j^\ast = H_\ell$, where $H_\ell$ is the last link in the chain $H_0 , H_1, \dots, H_\ell$, with $\epsilon= \epsilon_j$. We need to verify four claims.

{\bf Claim 1} We have
\[\Omega_1 \cup \Omega_2 \cup \dots \cup \Omega_\ell = \X \, .\]

Here each $\Omega_k$ is contained in $\X$, so we need only show that every $x\in \X$ lies in one of those domains. Suppose, to the contrary, that $x \not \in \Omega_1 \cup \dots \cup \Omega_\ell$. Since $x \not \in \Omega_\ell$ we see that $H_\ell (x)= H_{\ell -1} (x)$, and by induction that $H_{\ell -1}(x)= H_{\ell -2}(x)= \dots = H_0 (x)= h(x)$. These equalities, together with the definition of $\Omega_1, \dots , \Omega_\ell$ yield:
\[
\begin{split}
h(x) & = H_0 (x) \not \in \Y \cup \Upsilon_1 \\
h(x) & = H_1 (x) \not \in \Y \cup \Upsilon_2 \\
\vdots  & \hskip2cm \vdots  \\
h(x) & = H_{\ell -1} (x) \not \in \Y \cup \Upsilon_\ell
\end{split}
\]
Thus $h(x) \not \in \Y \cup \Upsilon_1 \cup \dots \cup \Upsilon_\ell = \overline{\Y}\,$, a clear contradiction.

{\bf Claim 2} $\,H_\ell \colon \X \onto \Y\,$. We have $\,H_\ell(\X)\supset H_\ell(\Omega_\ell)\supset \Y\,$. Now take $x\in \X$ and assume, to the contrary, that $H_\ell (x)\in \partial \Y$. Since $H_\ell \colon \Omega_\ell \onto \Y$ it follows that $x \not \in \Omega_\ell$ which in turn means that $H_\ell (x) = H_{\ell -1} (x) \in \partial \Y$. But $H_{\ell-1} \colon \Omega_{\ell -1} \onto \Y$, so $x \not \in \Omega_{\ell -1}$. Continuing in this way we conclude that $x\not \in \Omega_1$. This summarizes as  $x\not \in \Omega_1\cup \dots \cup \Omega_\ell= \X$, a clear contradiction.

{\bf Claim 3} $H_\ell \colon \X \onto \Y$ is injective.  Suppose, to the contrary, that $H_\ell (x_1)= H_\ell (x_2)=y \in \Y$ for some $x_1 \not = x_2$. Since $H_\ell \colon \Omega_\ell \onto \Y$ is a homeomorphism it is impossible that both $x_1$ and $x_2$ belong to $\Omega_\ell$. It is also impossible that $x_1 \in \Omega_\ell$ and $x_2 \not \in \Omega_\ell$; since otherwise we would have $H_\ell (x_1) \in \Y$ and $H_\ell (x_2) \in \partial \Y \setminus \Upsilon_\ell$. This leaves us with the only possibility that $x_1, x_2 \not \in \Omega_\ell$, so $y\in \partial \Y \setminus \Upsilon_\ell$, contradicting Claim 2.

{\bf Claim 4} Obviously, we have
\[\norm{h_j^\ast -h}_{\mathscr R^p (\X)}  \le \sum_{k=0}^\ell \norm{H_k -H_{k-1}}_{\mathscr R^p (\X)} \le \ell \, \epsilon_j  \to 0\, . \]

\end{proof}

\section{The case of simply connected domains}\label{seclas}
Theorem~\ref{thmain} fails when $\ell=1$ and $p=2$. To see this consider a sequence of M\"obius transformations $\,h_k \colon \mathbb D \onto \mathbb D\,$ of the unit disk onto itself,
\begin{equation}\label{eqmo}
h_k(z)= \frac{z+a_k}{1+ {a}_k z}\, , \qquad  a_k = \frac{k-1}{k} \to 1\, , \quad h_k (1)=1\, , \quad h_k (-1)=-1\,  .
\end{equation}
All the mappings have the same finite Dirichlet energy.
\[\mathcal E_{\mathbb D} [h_k] = \iint_{\mathbb D} \abs{Dh_k}^2 = 2\iint_{\mathbb D} J(z, h_k)\, \dtext z = 2 \pi \, .\]
They converge to the constant mapping $h(z) \equiv 1$, $c$-uniformly and weakly in $\W^{1,2} (\mathbb D)$. Obviously  $h(z)$ cannot be obtained as a strong limit in $\W^{1,2} (\mathbb D)$ of homeomorphisms of $\mathbb D$ onto itself. Nevertheless, continuous monotone extensions of $h_k$ outside $\mathbb D$ are still available, but  at the sacrifice of loosing uniform bounds (independent of $k$) of the Dirichlet  energy in any domain $\Omega \supset \overline{\mathbb D}$. Indeed, otherwise we would have a uniform bound for the modulus of continuity,  which in turn would imply that the boundary mappings $h_k \colon \partial \mathbb D \onto \partial \mathbb D$ converge uniformly; this is not the case. Thus we see that  reflecting $\mathbb D$ about the boundary circle cannot help.

\subsection{Proof of Theorem~\ref{thmain} for simply connected domains, $ p> 2 $ }
When the Sobolev exponent is greater than $2$ the homeomorphisms $h_k \in \Ho_p (\X, \Y)$ and their weak limit $h\in \widetilde{\Ho}_p (\X, \Y)$ enjoy uniform modulus of continuity, see Lemma \ref{lem212}. The reduction to the case of unit disks, $\,\X = \mathbb D\,$ and  $\,\Y = \mathbb D\,$, can be made via a bi-Lipschitz transformation as in Section \S\ref{sectcd}. This reduction is needed only to make the target domain  convex, so as to apply the $p$-harmonic replacements. The point is that there is no need to reflect $h_k$ or $h$ about the boundaries of $\X$ and $\Y$;  radial extensions outside the disks (for topological purposes) work just as well.  The rest of the proof of Theorem~\ref{thmain} runs as in the case $\ell \ge 2\,$, with hardly any changes.

\subsection{Simply connected domains with a puncture}
Let us now take on stage homeomorphisms $h_k \colon \X \onto \Y\,$, $\;h_k \in \Ho_2 (\X, \Y)\,,$ between simply connected Lipschitz domains. Assume that the mappings $h_k$ are fixed at a point $x_\circ \in \X$, namely.
\[h(x_\circ)=y_\circ \, , \qquad k=1,2, \dots , \;\;\;\;\textnormal{where}\;\; x_\circ \in \X \;\;\textnormal{and}\;\; y_\circ \in \Y\, .\]
Because of this normalization we have  an extension $\widehat{h}_k \colon \widehat{\C} \onto \widehat{\C}$ without loosing  uniform bounds in a neighborhood of $\overline{\X}$. Indeed, with the aid of  bi-Lipschitz transformation, see Section \S\ref{sectcd}\,, we are reduced to the mappings $h_k \colon \mathbb D \onto \mathbb D$ between unit disks such that $\,h_k (0)=0\,$. With this normalization the mappings  can be interpreted as  homeomorphisms $h_k \colon \mathbb D_\circ \onto \mathbb D_\circ\,$ between punctured disks $\mathbb D_\circ = \mathbb D \setminus \{0\}$. By Theorem 8.1 in~\cite{IO} we infer that
\begin{equation}\label{eqlss}
\dist^2\left[\,h_k (z)\,,\; \partial \mathbb D\,\right]\; \le \;\frac{C \,\cdot \, \mathcal E_{\mathbb D}[h_k]}{\log \left(e+ \frac{2}{\dist (z, \,\partial \mathbb D\,)}\right)}
\end{equation}
where $C$ is a universal constant. Now consider the extension $\widehat{h} \colon \widehat{\C} \to \widehat{\C}$, by the formula
\[\widehat{h}_k (z)= \left[\overline{h_k \left(1/\overline{z}\right)}\right]^{-1} \qquad \textnormal{ for } \abs{z} \ge 1\, . \]
We have $\,\mathcal E_{\mathbb D}[h_k] \le M\,$ for some constant $\,M < \infty\,$ and all $k=1,2, \dots\;$. Choose and fix $\,r>1\,$ sufficiently close to $1$ so that
\[\frac{C \cdot  M}{\log \left(e+ \frac{2r}{r-1}\right)} \le \frac{1}{4}\, .\]
Inequality~\eqref{eqlss} reads as
\[\big[\,1- \abs{h_k (1/\overline{z})} \,\big]^2 \le \frac{C\cdot \mathcal E_{\mathbb D}[h_k]}{\log \left(e+ \frac{2 \, \abs{z}}{\abs{z}-1}\right)} \le \frac{C \cdot M}{\log \left(e+ \frac{2r}{r-1}\right)} \le \frac{1}{4} \, .\]
Hence $\,\abs{\,h_k (1/ \overline{z})\,} \,\ge 1/2\,$, which gives us the desired uniform bound of the Dirichlet energy
\[\iint_{\mathbb D_r} \abs{\nabla \widehat{h}_k}^p  \le \;\widehat{C} \,\cdot\, \iint_{\mathbb D} \abs{\nabla {h}_k}^p\, , \qquad k=1,2, \dots\]
The remaining  arguments for the proof of Theorem~\ref{thmain} are the same as in the multiply connected  case. Let us summarize these findings by the following statement,
\begin{theorem}
Let $\X$ and $\Y$ be simply connected Lipschitz domains and let $h_k \colon \X \onto \Y$ be homeomorphisms in $\W^{1,2} (\X, \Y)$ normalized by $h_k (x_\circ)=y_\circ$, for some $x_\circ \in \X$ and $y_\circ \in \Y$, and all $k=1,2, \dots$ Suppose $h_k$ converge weakly in $\W^{1,2} (\X, \Y)$ to a mapping $h \in \W^{1,2} (\X, \Y)$. Then there exists a sequence of $\mathscr C^\infty$-diffeomorphisms
\[h_k^\ast \colon \X \onto \Y \, , \qquad h_k^\ast \in h+ \W^{1,2}_\circ (\X , \Y)\]
converging to $h$ strongly in $\W^{1,2} (\X, \Y)$. Moreover,
\begin{equation}
h_k^\ast (x_\circ)=y_\circ \qquad \textnormal{ for all } k=1,2, \dots
\end{equation}
\end{theorem}
\begin{remark}
The reader may be concerned that our proof of Theorem~\ref{thmain} does not guarantee that $h_k^\ast (x_\circ)=y_\circ$. However, since $h_k^\ast (x_\circ) \to h(x_\circ)=y_\circ$, one can alter slightly each $h_k^\ast$,  by a smooth change of variables in $\X$, to achieve such a normalization.
\end{remark}

\subsection{A normalization at three boundary points}\label{secnat}
Consider homeomorphisms $h_k \colon \X \to \Y$ between bounded simply connected Lipschitz domains in the Sobolev space $\W^{1,2}(\X, \Y)$, normalized by three conditions at the boundary:  $\,h(x_1)=y_1$, $h(x_2)=y_2$ and $h(x_3)=y_3\,,$ where $x_1, x_2, x_3$ are given distinct points in $\partial \X$ and $y_1, y_2, y_3 \in \partial \Y$. We assume that both triples $(x_1, x_2, x_3)$ and $(y_1, y_2, y_3)$ are positively oriented along $\partial \X$ and $\partial \Y$, respectively. We transform the domains $\X$ and $\Y$ via bi-Lipschitz mappings into equilateral triangles in which $(x_1, x_2, x_3)$ and $(y_1, y_2, y_3)$ are their vertices.  As in the multiply connected  case, we will need to extend the mappings to a neighborhood $\,\mathbb X_+ \supset \overline{\mathbb X}\,$. To build this neighborhood, we reflect $\,\mathbb X\,$ about its sides. This results in three adjacent triangles. We then reflect the adjacent triangles. Eventually we come up with 12 equilateral triangles surrounding $\,\mathbb X\,$ which together with $\,\overline{\mathbb X}\,$ make the neighborhood $\,\mathbb X_+\,$.
The same reflection procedure is  made for $\Y$. We then extend each homeomorphism $\,h_k \colon \X \onto \Y\,$ in accordance with the above reflections  and define continuous monotone extensions $\widehat{h}_k \colon \X_+ \onto \Y_+$. The way the extensions are made secures a uniform energy bound
\[ \iint_{\X_+} \abs{\nabla \widehat{h}_k}^2 = 13\,  \iint_{\X} \abs{\nabla \widehat{h}_k}^2  \le M \;,\;\;\textnormal{for some}\; M \; \; \textnormal{and all }\;\;\; k=1,2, \dots\]
By the same arguments as for  Theorem~\ref{thmain}  we draw the following conclusion
\begin{theorem}
Let $\X$ and $\Y$ be simply connected Lipschitz domains and let $h_k \colon \X \onto \Y$ be homeomorphisms of Sobolev class $\W^{1,2}(\X, \Y)$. \footnote{Note that continuous extensions $h_k \colon \overline{\X} \onto \overline{\Y}$, $k=1,2, \dots$, exist automatically.}
Suppose $h_k$ are normalized at three boundary points;
\[h(x_1)=y_1\, , \quad h(x_2)=y_2\, , \quad h(x_3)=y_3\]
and that $h_k$ converge weakly in $\W^{1,2} (\X, \Y)$ to a mapping $h\in \W^{1,2}(\X, \Y)$. Then there exist  $\mathscr C^\infty$-diffeomorphisms
\[h_k^\ast \colon \X \onto \Y \, , \qquad h_k^\ast \in h+ \W^{1,2}_\circ (\X , \Y)\]
converging to $h$ strongly in $\W^{1,2} (\X, \Y)$ and uniformly on $\overline{\X}$.
\end{theorem}

 \bibliographystyle{amsplain}

\begin{thebibliography}{29}


\bibitem{Ahb}
L. V. Ahlfors,  \textit{Lectures on quasiconformal mappings}, Second edition. American Mathematical Society, Providence, RI, 2006.


\bibitem{AS}
G. Alessandrini and M. Sigalotti, \textit{Geometric properties of solutions to the anisotropic $p$-Laplace equation in dimension two},
Ann. Acad. Sci. Fenn. Math. {\bf 26} (2001), no. 1, 249--266.

\bibitem{AIMb}
K. Astala, T. Iwaniec, and G. Martin, \textit{Elliptic partial differential equations and quasiconformal mappings in the plane},  Princeton University Press, Princeton, NJ, 2009.


\bibitem{AIM}
K. Astala, T. Iwaniec, and G. Martin, \textit{Deformations of annuli with smallest mean distortion}, Arch. Ration. Mech. Anal. {\bf 195} (2010), no. 3, 899--921.

\bibitem{AIS}
K. Astala,  T.  Iwaniec, and E.  Saksman,  \textit{Beltrami operators in the plane},  Duke Math. J. {\bf 107} (2001), no. 1, 27--56.



\bibitem{Bac}
J. M. Ball, \textit{Convexity conditions and existence theorems in nonlinear elasticity},  Arch. Rational Mech. Anal. {\bf 63} (1976/77), no. 4, 337--403.


\bibitem{Ba3}
J. M. Ball, \textit{Constitutive inequalities and existence theorems in nonlinear elastostatics},  Nonlinear analysis and mechanics: Heriot-Watt Symposium (Edinburgh, 1976), Vol. I, pp. 187Ð241. Res. Notes in Math., No. 17, Pitman, London, (1977).

\bibitem{Ba1}
J. M. Ball, \textit{Discontinuous equilibrium solutions and cavitation in nonlinear elasticity}, Philos. Trans. R. Soc. Lond. A {\bf 306} (1982),  557--611.

\bibitem{Ba4}
J.  M. Ball,  \textit{Minimizers and the Euler-Lagrange equations}, Trends and applications of pure mathematics to mechanics (Palaiseau, 1983), 1--4, Lecture Notes in Phys., 195, Springer, Berlin, 1984.



\bibitem{Ba}
J. M. Ball,  \textit{Singularities and computation of minimizers for variational problems}, Foundations of computational mathematics (Oxford, 1999), 1--20, London Math. Soc. Lecture Note Ser., 284, Cambridge Univ. Press, Cambridge, 2001.

\bibitem{Ba5}
J.  M. Ball,  \textit{Some open problems in elasticity},  Geometry, mechanics, and dynamics, 3--59, Springer, New York, 2002.

\bibitem{Ba2}
J. M. Ball, \textit{Progress and puzzles in nonlinear elasticity},
in ``Poly-, quasi- and rank-one convexity in applied mechanics'' (eds. J. Schr\"oder and P. Neff),
Proceedings of CISM International Centre for Mechanical Sciences, vol.~516 (2010), 1--15.


\bibitem{BOP}
P. Bauman, N. C. Owen, and D. Phillips, \textit{Maximal smoothness of solutions to certain Euler-Lagrange equations from nonlinear elasticity}, Proc. Roy. Soc. Edinburgh Sect. A {\bf 119} (1991), no. 3-4, 241--263.

\bibitem{BPO1}
P. Bauman,  D. Phillips, and N. Owen,  \textit{Maximum principles and a priori estimates for an incompressible material in nonlinear elasticity}, Comm. Partial Differential Equations {\bf 17} (1992), no.~7-8, 1185--1212.

\bibitem{BM}
J. C. Bellido and C. Mora-Corral, \textit{Approximation of H\"older continuous homeomorphisms by piecewise affine homeomorphisms}, Houston J. Math., 37, no. 2 (2011) 449--500.


\bibitem{CLMS}
R. Coifman, P.-L. Lions,  Y.  Meyer,  and S.  Semmes, \textit{Compensated compactness and Hardy spaces},  J. Math. Pures Appl. (9) {\bf 72} (1993), no. 3, 247--286.

\bibitem{BI}
B.  Bojarski and T. Iwaniec,  \textit{$p$-harmonic equation and quasiregular mappings},  Partial differential equations (Warsaw, 1984), 25Ð38, Banach Center Publ., 19, PWN, Warsaw, 1987.

\bibitem{CL}
S. Conti and C. De Lellis, \textit{Some remarks on the theory of elasticity for compressible Neohookean materials}, Ann. Sc. Norm. Super. Pisa Cl. Sci.  (5) {\bf 2} (2003) 521--549.

\bibitem{Cob}
R. Courant, \textit{Dirichlet's principle, conformal mapping, and minimal surfaces},
With an appendix by M. Schiffer.  Springer-Verlag, New York-Heidelberg, 1950.

\bibitem{CIKO}
J. Cristina, T. Iwaniec, L. V. Kovalev and J. Onninen,
\textit{Lipschitz regularity for the Hopf-Laplace equation}, arXiv:1011.5934.

\bibitem{DP}
S. Daneri and A. Pratelli, \textit{Smooth approximation of bi-Lipschitz orientation-preserving homeomorphisms},  arXiv:1106.1192.

\bibitem{DI}
L. D'Onofrio and T. Iwaniec, \textit{Notes on $p$-harmonic analysis. The p-harmonic equation and recent advances in analysis}, 25Ð49, Contemp. Math., 370, Amer. Math. Soc., Providence, RI, 2005.


\bibitem{Dub}  P. Duren,  \textit{Harmonic mappings in the plane}, Cambridge Tracts in Mathematics, 156. Cambridge University
Press, Cambridge, 2004.

\bibitem{Ev}
L. C. Evans, \textit{Quasiconvexity and partial regularity in the calculus of variations}, Arch. Rational Mech. Anal. {\bf 95} (1986), no. 3, 227--252.




\bibitem{Ha}
P. Haj\l asz, \textit{Pointwise Hardy inequalities}, Proc. Amer. Math. Soc. \textbf{127} (1999), no.~2, 417--423.

\bibitem{Ha1}
P. Haj\l asz, \textit{Sobolev mappings: Lipschitz density is not a bi-Lipschitz invariant of the target}, Geom. Funct. Anal. {\bf 17} (2007), no. 2, 435Ð467.


\bibitem{HKM}
J. Heinonen,  T.  Kilpel\"ainen and J.  Mal\'y, \textit{J. Connectedness in fine topologies}, Ann. Acad. Sci. Fenn. Ser. A I Math. {\bf 15} (1990), no. 1, 107--123.

\bibitem{HKMb}
J. Heinonen, T. Kilpel\"ainen and O. Martio, \textit{Nonlinear potential theory of degenerate elliptic equations}, Oxford University Press, New York, 1993.

\bibitem{IKKO}
T. Iwaniec, N.-T. Koh, L. V. Kovalev, and J. Onninen,
\textit{Existence of energy-minimal diffeomorphisms between doubly connected domains}, Invent. Math. (2011), to appear.





\bibitem{IKO2}
T. Iwaniec, L. V. Kovalev  and J. Onninen, \textit{Diffeomorphic approximation of Sobolev homeomorphisms} Arch. Rat. Mech. Anal., {\bf 201} no. 3, 1047--1067.

\bibitem{IKO3}
T. Iwaniec, L. V. Kovalev  and J. Onninen, \textit{Approximation up to the boundary of homeomorphisms of finite Dirichlet energy},  Bull. Lond. Math. Soc., to appear.

\bibitem{IKO4}

T. Iwaniec, L. V. Kovalev  and J. Onninen, \textit{Lipschitz regularity for inner-variational equations},  arXiv:1109.0720.

\bibitem{IM}
T. Iwaniec and J. J. Manfredi,  \textit{Regularity of $p$-harmonic functions on the plane}, Rev. Mat. Iberoamericana {\bf 5} (1989), no. 1-2, 1--19.



\bibitem{IOh}
T. Iwaniec and J. Onninen, \textit{$\mathcal H^1$-estimates of Jacobians by subdeterminants},  Math. Ann. {\bf 324} (2002), no. 2, 341--358.

\bibitem{IO}
T. Iwaniec and J. Onninen, \textit{Deformations of finite conformal energy: Boundary behavior and limit theorems},  Trans. Amer. Math. Soc. {\bf 363} (2011), no.~11,  5605--5648.

\bibitem{IO*}
T. Iwaniec and J. Onninen, \textit{Mappings of least Dirichlet energy and their Hopf diffferentials},  preprint.


\bibitem{Jo}
J. Jost, \textit{A note on harmonic maps between surfaces},  Ann. Inst. H. Poincar\'e Anal. Non Lin\'eaire \textbf{2} (1985), no.~6, 397--405.

\bibitem{Job}
J.  Jost, \textit{Two-dimensional geometric variational problems}, John Wiley \& Sons, Ltd., Chichester, 1991.


\bibitem{JS}
J. Jost and R. Schoen, \textit{On the existence of harmonic diffeomorphisms}, Invent. Math. {\bf 66} (1982), no. 2, 353--359.

\bibitem{KM}
T. Kilpel\"ainen and J.  Mal\'y, \textit{Degenerate elliptic equations with measure data and nonlinear potentials},  Ann. Scuola Norm. Sup. Pisa Cl. Sci. (4) {\bf 19} (1992), no. 4, 591--613.

\bibitem{Ku}
K. Kuratowski, \textit{On the completeness of the space of monotone mappings and some related problems},  Bull. Acad. Polon. Sci. S\'er. Sci. Math. Astronom. Phys. {\bf 16} (1968) 283--285.

\bibitem{KL}
K. Kuratowski and R. C.  Lacher, \textit{A theorem on the space of monotone mappings},
Bull. Acad. Polon. Sci. S\'er. Sci. Math. Astronom. Phys. {\bf 17} (1969) 797--800.



\bibitem{Le}
J. Lehrb\"ack, \textit{Pointwise Hardy inequalities and uniformly fat sets}, Proc. Amer. Math. Soc. \textbf{136} (2008), no.~6, 2193--2200.

\bibitem{MM}
M. Marcus and V. J. Mizel, \textit{Every superposition operator mapping one Sobolev space into another is continuous} J. Funct. Anal. {\bf 33} (1979), no. 2, 217--229.

\bibitem{Mc}
L. F. McAuley,  \textit{Some fundamental theorems and problems related to monotone mappings} 1971 Proc. First Conf. on Monotone Mappings and Open Mappings (SUNY at Binghamton, Binghamton, N.Y., 1970). 1--36.  State Univ. of New York at Binghamton, N.Y.


\bibitem{Mo}
C. Mora-Corral, \textit{Approximation by piecewise affine homeomorphisms of Sobolev homeomorphisms that are smooth outside a point},  Houston J. Math. \textbf{35} (2009), no.~2, 515--539.

\bibitem{Mor}
C. B. Morrey,  \textit{The Topology of (Path) Surfaces}, Amer. J. Math. {\bf 57} (1935), no. 1, 17--50.

\bibitem{Moq}
C. B.  Morrey, \textit{Quasi-convexity and the lower semicontinuity of multiple integrals},  Pacific J. Math. {\bf 2}, (1952). 25--53.


\bibitem{Mu}
S.  M\"uller,  \textit{Higher integrability of determinants and weak convergence in $L^1$}, J. Reine Angew. Math. {\bf 412} (1990), 20--34.





\bibitem{SSe}
E. Sandier and S.  Serfaty,  \textit{Limiting vorticities for the Ginzburg-Landau equations}, Duke Math. J. {\bf 117} (2003), no.~3, 403--446.


\bibitem{SeS}
G. A.  Seregin and T. N. Shilkin,   \textit{Some remarks on the mollification of piecewise-linear homeomorphisms}.  J. Math. Sci. (New York) \textbf{87} (1997), no.~2, 3428--3433.

\bibitem{Ra}
T. Rad\'o, \textit{On continuous mappings of Peano spaces},
Trans. Amer. Math. Soc. {\bf 58}, (1945). 420--454.

\bibitem{Rab}
T. Rad\'o, \textit{Length and Area},  American Mathematical Society, New York, 1948.

\bibitem{SiSp}
J. Sivaloganathan and S. J. Spector,  \textit{Necessary conditions for a minimum at a radial cavitating singularity in nonlinear elasticity}, Ann. Inst. H. Poincar\'e Anal. Non Lin\'eaire {\bf 25} (2008), no. 1, 201--213.

\bibitem{SS}
J. Sivaloganathana and S. J. Spector, \textit{On irregular weak solutions of the energy-momentum equations},
Proc.  R. Soc.  Edinb. A \textbf{141} (2011), 193--204.

\bibitem{Ta}
A. Taheri,  \textit{Quasiconvexity and uniqueness of stationary points in the multi-dimensional calculus of variations}, Proc. Amer. Math. Soc. \textbf{131} (2003), no.~10, 3101--3107.


\bibitem{Whb}
G. T. Whyburn, \textit{Analytic topology},  American Mathematical Society, Providence, R.I. (1963).

\bibitem{Ya}
X. Yan, \textit{Maximal smoothness for solutions to equilibrium equations in 2D nonlinear elasticity}, Proc. Amer. Math. Soc. {\bf 135} (2007), no. 6, 1717--1724.


\bibitem{Yo}
J. W. T. Youngs, \textit{The topological theory of Fr\'echet surfaces},
Ann. of Math. (2) \textbf{45} (1944), 753--785.

\bibitem{Yo1}
J. W. T. Youngs, \textit{Homeomorphic approximations to monotone mappings},  Duke Math. J. {\bf 15}, (1948). 87--94.




\end{thebibliography}

 \end{document}